\documentclass[a4paper]{amsart}
\usepackage{amsmath}
\usepackage{cases}
\usepackage{amsfonts}
\usepackage[colorlinks,linkcolor=blue,citecolor=blue]{hyperref}
\usepackage{latexsym, amssymb, amsmath, amsthm, bbm}
\usepackage[all]{xy}

\usepackage{anysize}\marginsize{30mm}{30mm}{30mm}{30mm}
\def \Id{\operatorname{Id}}

\def \Ext{\operatorname{Ext}}

\def \id{\operatorname{Id}}
\def \ker{\operatorname{Ker}}

\def \im{\operatorname{im}}
\def \io{\operatorname{io}}

\def \D{\Delta}

\numberwithin{equation}{section}

\newtheorem{theorem}{Theorem}[section]
\newtheorem{lemma}[theorem]{Lemma}
\newtheorem{proposition}[theorem]{Proposition}
\newtheorem{corollary}[theorem]{Corollary}
\newtheorem{definition}[theorem]{Definition}

\newtheorem{remark}[theorem]{Remark}

\begin{document}
\title{Classification of affine prime regular Hopf algebras of GK-dimension one}\thanks{$^\dag$Supported by NSFC 11371186, 11371187, 11571164.}


\author{Jinyong Wu}
\address{Department of Mathematics, Nanjing University, Nanjing 210093,
China}
\address{Yiwu Industrial and Commercial College, Yiwu,
Zhejiang 322000, China} \email{wujy1753@126.com}

\author{Gongxiang Liu}
\address{Department of Mathematics, Nanjing University, Nanjing 210093, China} \email{gxliu@nju.edu.cn}

\author{Nanqing Ding}
\address{Department of Mathematics, Nanjing University, Nanjing 210093, China} \email{nqding@nju.edu.cn}

\date{}
\maketitle

\begin{abstract}
The classification of affine prime regular Hopf algebras of GK-dimension one
is completed. As consequences, 1) we give a negative answer to Question 7.1 posed
in \cite{BZ} and 2) we show that there do exist prime regular Hopf algebras of GK-dimension one
which are not pointed.
\end{abstract}

\section{Introduction}
Throughout this paper, $k$ denotes an algebraically closed field
of characteristic $0$, all vector spaces are over $k$, and all Hopf
algebras are affine and noetherian.

In recent years, Hopf algebras of finite Gelfand-Kirillov dimension
(GK-dimension for short) have been studied intensively \cite{AA, AS,
BZ, GZ, Liu, LWZ, WZZ1, WZZ2, WZZ3}. In particular, the
classification of Hopf algebras of low GK-dimension was carried over
and got substantial progress. Lu, Wu and Zhang initiated the the
program of classifying Hopf algebras of GK-dimension one \cite{LWZ}.
Brown and Zhang \cite{BZ} made further efforts in this direction and
classified all prime regular Hopf algebras $H$ of GK-dimension one
under the hypothesis:
 $$\im(H)=1\;\;\;\;\;\text{or}
\;\;\;\;\;\im(H)=\io(H)\;\;\;\;\;\;\;\;\;(\ast)$$ (see Section $2$
for the definition of $\im(H)$ and $\io(H)$).

For Hopf algebras $H$ of GK-dimension two, all known classification
results are given under the condition of $H$ being domains. In
\cite{GZ}, Goodearl and Zhang classified all Hopf algebras $H$ of
GK-dimension two which are domains and satisfy the condition
$\Ext^{1}_{H}(k, k)\neq 0$. For those with vanishing Ext-groups, some
interesting examples were constructed by Wang-Zhang-Zhuang
\cite{WZZ2} and they conjectured these examples together with Hopf
algebras given in \cite{GZ} exhausted all Hopf algebra domains with
GK-dimension two. In order to study Hopf algebras $H$ of
GK-dimensions three and four, a more restrictive but natural
condition was added: $H$ is connected, that is, the coradical of $H$
is $1$-dimensional.  All connected Hopf algebras with GK-dimension
three were classified by Zhuang in \cite{Zh}. Recently, Wang, Zhang
and Zhuang \cite{WZZ3} completed the classification of connected
Hopf algebras of GK-dimension four.

In another direction, the pointed Hopf algebra domains of finite
GK-dimension with generic infinitesimal braiding have been
classified by  Andruskiewitsch and Schneider \cite{AS} and
Andruskiewitsch and Angiono \cite{AA}.

Philosophically, prime regular Hopf algebras of finite
GK-dimension can be regarded as ``noncommutative" counterparts of
connected algebraic groups. It is well-known that there are only two connected algebraic groups of
dimension one. This fact makes us believe that
 there should be a complete classification of prime regular Hopf algebras
 of GK-dimension one. In this paper, we will go on with Brown-Zhang's works \cite{BZ} and
finish the classification of all
prime regular Hopf algebras of GK-dimension one.

As the first step,
we construct a new class of prime regular Hopf algebras $D(m,d,\xi)$ of
GK-dimension one,
which implies that
the condition $(*)$ is really necessary for Brown-Zhang's
classification.
In fact, Brown and Zhang \cite{BZ} gave four classes of Hopf
algebras of GK-dimension one: the coordinate algebras of connected
algebraic groups of dimension one, the infinite dihedral group algebra, infinite
dimensional Taft algebras and generalized Liu algebras, and proved
that all prime regular Hopf algebras $H$ of GK-dimension one
satisfying condition $(*)$ belong to these four classes. Naturally,
they raised the following open question (Question 7.1 in \cite{BZ}):

(\textbf{Q}) \emph{Does their result still hold without the hypothesis $(*)$}?\\[1.5mm]
Our new examples in Section 4 of this paper provide a negative answer to this
question.

Secondly, we will prove our main result (see Theorem \ref{t8.3}) which states that  our new examples
 together
with the four classes of Hopf algebras given in \cite{BZ} form a complete list,
up to isomorphisms of Hopf algebras,
of all prime regular Hopf algebras
of GK-dimension one. The key idea to prove the main result is not complicated:
let $H$ be a
prime regular Hopf algebra of GK-dimension one which doesn't
satisfy the condition $(*)$. From this Hopf algebra $H$, we can
construct a Hopf subalgebra $\widetilde{H}$ which will be shown to
meet the condition $(*)$. Thus the classification result given in
\cite{BZ} can be applied. At last, we show that $\widetilde{H}$
determines the structure of $H$ entirely.

The precess to realize our idea, which motivates our discovery of new examples,
 turns out  to be much more complicated than our expectation:
According to Brown-Zhang's classification result, there is a dichotomy on $\widetilde{H}$:
$\widetilde{H}$ is either primitive or group-like (see Definition 6.1).  When $\widetilde{H}$ is
primitive, we find that $H$ must be an infinite dimensional Taft algebra. The difficult part is group-like
case. In this case, we gradually realize that there is an essential difference
 between the situation $\frac{\io(H)}{\im{(H)}}>2$ and the situation $\frac{\io(H)}{\im{(H)}}=2$.
 The last situation becomes very delicate: More generators and subtle relations are allowed to appear.
 Ultimately, this leads us to find the final missing piece in the puzzle of prime regular
 Hopf algebras of GK-dimension one.

In practice, the assumption ``pointed" is always added  when we want to
classify Hopf algebras of lower GK-dimensions.
As a matter of fact, all known examples are pointed
and it is widely believed that, at least for prime regular Hopf algerbas of GK-dimension one, these
Hopf algebras should be pointed automatically. Our new examples will change this naive understanding
since all the new examples are not pointed!

The paper is organized as follows. Necessary definitions,
known examples and preliminary results are collected in Section 2.
Some combinatorial relations, which are
crucial to the following analysis, will be given in Section 3. Section 4 is
 devoted to constructing the new examples $D(m,d,\xi)$ of prime regular
Hopf algebras of GK-dimension one. The proof of $D(m,d,\xi)$ being a
Hopf algebra is nontrivial and the combinatorial
equations given in Section 3 are used extensively in the proof. In Section 5,
the definition of the Hopf subalgebra $\widetilde{H}$ and some basic properties of
$\widetilde{H}$ are given.
In particular, we show that  $\widetilde{H}$ is a prime regular Hopf algebra satisfying the condition
$(*)$. Thus $\widetilde{H}$ is either primitive or group-like. By the results of the last three sections, we can reconstruct $H$ from $\widetilde{H}$. In details, the general relations between $\widetilde{H}$
and $H$ are built in Section 6. Section 7 is designed to analyse the primitive case and
$H$ is shown to be an infinite dimensional Taft algebra. We consider the group-like case in the last section, and
as the desired conclusion, we show that $H$ is isomorphic to some $D(m,d,\xi)$. Finally, the main result, that is, the classification result, and its proof are also formulated in this section.

\section{Preliminaries}

In this section we recall the urgent needs around affine noetherian
Hopf algebras for completeness and the convenience of the reader.
About general background knowledge, the reader is referred to
\cite{Mo} for Hopf algebras, \cite{MR} for noetherian rings and
\cite{Br,LWZ,BZ} for exposition about noetherian Hopf algebras.

Usually we are working on left modules. Let $A^{op}$ denote the
opposite algebra of $A$. Throughout, we use the symbols $\Delta,
\epsilon$ and $S$ respectively, for the coproduct, counit and
antipode of a Hopf algebra $H$, and the Sweedler's notation for
coproduct $\D(h)=\sum h_1\otimes h_2\;(h\in H)$ will be used freely.

 \subsection{}\textbf{Stuffs from ring theory.}

In this paper, a ring $R$ is called \emph{regular} if it has finite global dimension,
and it is \emph{prime} if $0$ is a prime ideal.  \\[2mm]
{$\bullet$ \emph{PI-degree}.}  If $Z$ is an Ore domain, then the
{\it rank } of a $Z$-module $M$ is defined to be the
$Q(Z)$-dimension of $Q(Z)\otimes_Z M$, where $Q(Z)$ is the quotient
division ring of $Z$.  Let $R$ be an algebra satisfying a polynomial
identity (PI for short). The PI-degree of $R$ is defined to be
$$\text{PI-deg}(R)=\text{min}\{n|R\hookrightarrow M_n(C)\ \text{for some commutative ring}\ C\}$$
(see \cite[Chapter 13]{MR}). If $R$ is a prime PI ring with center
$Z$, then the PI-degree of $R$ equals the square root of the rank of
$R$ over $Z$.\\[2mm]
{$\bullet$\emph{ Artin-Schelter condition}.}  Recall that an algebra $A$
is said to be \emph{augmented} if there is an algebra morphism
$\epsilon:\; A\to k$. Let $(A,\epsilon)$ be an augmented noetherian
algebra. Then
$A$ is \emph{Artin-Schelter Gorenstein}, we usually abbreviate to \emph{AS-Gorenstein}, if \\[1.5mm]
\indent (AS1) injdim$_AA=d<\infty$,\\
\indent (AS2) dim$_k\Ext_A^d(_Ak, \;_AA)=1$ and
dim$_k\Ext_A^i(_Ak,\;
_AA)=0$ for all $i\neq d$,\\
\indent (AS3) the right $A$-module versions of (AS1, AS2) hold.

The following result is the combination of \cite[Theorem 0.1]{WZ} and  \cite[Theorem 0.2  (1)]{WZ},
which shows that a large number of Hopf algebras are AS-Gorenstein.
\begin{lemma}\label{l2.1} Each affine noetherian  PI Hopf algebra is AS-Gorenstein.
\end{lemma}

\subsection{}{\bf Homological integrals.}

The concept \emph{homological integral} can be defined for an
AS-Gorenstein augmented algebra.
\begin{definition}\cite[Definition 1.3]{BZ}
Let $(A, \epsilon)$ be a noetherian augmented algebra and suppose that
$A$ is  AS-Gorenstein of injective dimension $d$. Any non-zero
element of the one-dimensional $A$-bimodule $\Ext_A^d( _Ak,\; _AA)$ is
called a \emph{left homological integral} of $A$. We write
$\int_A^l=\Ext_A^d(_Ak,\; _AA)$. Any non-zero element in
$\Ext_{A^{op}}^d(k_A, A_A)$ is called a \emph{right homological
integral} of $A$. We write $\int_A^r=\Ext_{A^{op}}^d(k_A, A_A)$. By
abusing the language we also call $\int_A^l$ and $\int_A^r$ the left
and the right homological integrals of $A$ respectively.
\end{definition}

\noindent{$\bullet$ \emph{Winding automorphisms}.}  Let $H$ be an
affine noetherian PI Hopf algebra. By Lemma \ref{l2.1}, it is AS-Gorenstein and thus has
left homological integrals $\int_{H}^{l}$. Let $\pi : H \to H/\text{r.ann}(\int_H^l)$ be the canonical
algebra homomorphism, where
$\text{r.ann}(\int_H^l)$ denotes the set of right  annihilators of
$\int_H^l$ in $H$. We write $\Xi_\pi^l$ for the \emph{left winding
automorphism} of $H$ associated to $\pi$, namely
$$\Xi_\pi^l(a):=\sum\pi(a_1)a_2 \;\;\;\;\;\;\;\textrm{for} \;a\in H.$$ Similarly we use $\Xi_\pi^r$
for the right winding automorphism of $H$ associated to $\pi$,
that is,
$$\Xi_\pi^r(a):=\sum a_1\pi(a_2)\;\;\;\;\;\;\textrm{for}\; a\in H.$$

Let $G_\pi^l$ and $G_\pi^r$ be the subgroups of
$\text{Aut}_{k\text{-alg}}(H)$ generated by $\Xi_\pi^l$ and
$\Xi_\pi^r$, respectively.\\[2mm]
\emph{$\bullet$ Integral order and integral minor.} With the same
notions as above, the {\it integral order} $\io(H)$ of $H$ is
defined by the order of the group $G_\pi^l$ :
\begin{equation}\io(H):=|G_\pi^l|.\end{equation}
As noted in \cite[Lemma 2.1]{BZ}, we always have
$|G_\pi^l|=|G_\pi^r|$. So the above definition is independent of the choice
of $G_\pi^l$ or $G_\pi^r$. In addition, if $H$ is prime regular of
GK-dimension one, then \cite[Theorem 7.1]{LWZ} implies
$$\text{PI-deg}(H)=\io(H).$$

The \emph{integral minor} of $H$, denoted by $\im(H)$, is defined by
\begin{equation}\text{im}(H):=|G_\pi^l/G_\pi^l\cap G_\pi^r|.\end{equation}

\begin{remark} Crudely speaking, $\io(H)$ is a measure of the commutativity of $H$
and $\im(H)$ is a measure of the cocommutativity of $H$. In fact,
for a prime regular Hopf algebra $H$ of GK-dimension one, we have
 $\io(H)=1$ if and only if $H$ is
commutative \emph{(see }\cite[Corollary 7.8]{LWZ}\emph{)} and
$\im(H)=1$ if and only if $H$ is cocommutative \emph{(see
}\cite[Section 4]{BZ}\emph{)}.
\end{remark}

\noindent$\bullet$ \emph{Strongly graded and bigraded properties}.
Let $H$ be a prime regular Hopf algebra of GK-dimension one. By \cite[Theorem 2.5]{BZ},
$|G_\pi^l|$ is finite, say $n$. Therefore, the character group $\widehat{G_\pi^l}:=\text{Hom}_{k\text{-alg}}(kG_\pi^l, k)$ of $G_\pi^l$ is isomorphic to $G_\pi^l$. Similarly, the character group $\widehat{G_\pi^r}$ of
$G_\pi^r$ is isomorphic to $G_\pi^r$.


Fix a primitive $n$th root $\zeta$ of 1 in $k$, and define $\chi\in
\widehat{G_\pi^l}$ and $\eta\in \widehat{G_\pi^r}$ by setting
$$\chi(\Xi_\pi^l)=\zeta \quad \text{and} \quad
\eta(\Xi_\pi^r)=\zeta.$$ Thus $\widehat{G_\pi^l}=\{\chi^i|0\leqslant
i\leqslant n-1\}$ and $ \widehat{G_\pi^r}=\{\eta^j|0\leqslant
j\leqslant n-1\}$. For each $0\leqslant i, j\leqslant n-1$, let
$$H_i^l:=\{a\in H|\Xi_\pi^l(a)=\chi^i(\Xi_\pi^l)a\}$$ and
$$H_j^r:=\{a\in H|\Xi_\pi^r(a)=\eta^j(\Xi_\pi^r)a\}.$$

The  following lemma is  \cite[Theorem 2.5 (b)]{BZ}.

\begin{lemma} \emph{(1)} $H=\bigoplus_{\chi^i\in
\widehat{G_\pi^l}}H_i^l$ is strongly $\widehat{G_\pi^l}$-graded.

\emph{(2)} $H=\bigoplus_{\eta^j\in
\widehat{G_\pi^r}}H_j^r$ is strongly $\widehat{G_\pi^r}$-graded.
\end{lemma}

It is clear that $\Xi_\pi^l \Xi_\pi^r=\Xi_\pi^r \Xi_\pi^l$, so
$H_i^l$ is stable under the action of $G_\pi^r$. Consequently, the
$\widehat{G_\pi^l}$- and $\widehat{G_\pi^r}$-gradings on $H$ are
\emph{compatible} in the sense that
$$H_i^l=\bigoplus_{0\leqslant j \leqslant n-1}(H_i^l\cap H_j^r)\quad \text{and}\quad H_j^r=\bigoplus_{0\leqslant i \leqslant n-1}(H_i^l\cap H_j^r)$$
for all $i, j$.
 Then $H$ is a bigraded algebra:
\begin{equation}H=\bigoplus_{0\leqslant i,j\leqslant n-1} H_{ij},\end{equation}
where $H_{ij}=H_i^l\cap
H_j^r$. And we write $H_0=H_{00}$ for convenience.

For later use, we collect some properties about $H$ which are \cite[Proposition 2.1 (c)(e)]{BZ} and \cite[Theorem 2.5 (f)]{BZ}.
\begin{lemma}\label{l2.5} Let $H$ be an affine prime regular Hopf
algebra of GK-dimensional one. Then

\emph{(a)} $\Delta(H_i^l)\subseteq H_i^l\otimes H$ and $\Delta
(H_j^r)\subseteq H\otimes H_j^r$; thus $H_i^l$ is a right coideal of
$H$ and $H_j^r$ is a left coideal of $H$;

\emph{(b)} $\Xi_\pi^r S=S(\Xi_\pi^l)^{-1},$
where $(\Xi_\pi^l)^{-1}=\Xi_{\pi\circ S}^l.$

\emph{(c)} $H_0^l, H_0^r$ and $H_0$ are affine commutative Dedekind
domains with $H_0^l\cong H_0^r$.
\end{lemma}
Note that only (c) needs all the hypotheses of $H$, (a) and (b) hold
if $H$ has finite integral order.

\begin{remark}
\emph{(1)} By \cite{Sm,SSW}, prime affine algebras of GK-dimension one are
noetherian and PI automatically. So ``\emph{noetherian}" and ``\emph{PI}" do not appear in
the title of this paper.

\emph{(2)} If $H$ is an affine prime regular Hopf algebra of GK-dimensional
one, then \emph{gl.dim}$H=1$. Indeed, assume that \emph{gl.dim}$H=d$. Wu
and Zhang \cite{WZ} proved that every noetherian affine PI Hopf algebra is
Cohen-Macaulay and this forces $d=1$.

\end{remark}

\subsection{}{\bf Known examples.}

The following examples appeared in \cite{BZ} already and we recall them for completeness.

\noindent$\bullet$ \emph{Connected algebraic groups of dimension
one}. It is well-known that there are precisely two connected algebraic groups
of dimension one (see, say \cite[Theorem 20.5]{Hu}) over an algebraically
closed field $k$. Therefore, there are precisely two commutative
$k$-affine domains of GK-dimension one which admit a structure of
Hopf algebra, namely $H_1=k[x]$ and $H_2=k[x^{\pm 1}]$. For $H_1$,
$x$ is a primitive element, and for $H_2$, $x$ is a group-like
element. Commutativity and cocommutativity imply that
$\io(H_i)=\im(H_i)=1$ for $i=1, 2$.

\noindent$\bullet$ \emph{Infinite dihedral group algebra}. Let $\mathbb{D}$
denote the infinite dihedral group $\langle g, x | g^2 = 1,
gxg=x^{-1}\rangle$. Both $g$ and $x$ are group-like elements in the group
algebra $k\mathbb{D}$. By cocommutativity, $\im(k\mathbb{D})=1$.
Using \cite[Lemma 2.6]{LWZ}, one sees that as a right $H$-module,
$$\int_{k\mathbb{D}}^l\cong k\mathbb{D}/\langle x-1, g+1\rangle.$$
This implies $\io(k\mathbb{D})=2$.

\noindent$\bullet$ \emph{Infinite dimensional Taft algebras}. Let
$n$ and $t$ be integers with $n>1$ and
$0\leqslant t \leqslant n-1$. Fix a primitive $n$th root $\xi$ of
$1$. Let $H(n, t, \xi)$ be the algebra generated by $x$ and $g$
subject to the relations $$g^n=1\quad \text{and} \quad xg=\xi gx.$$
Then $H(n, t, \xi)$ is a Hopf algebra with coalgebra structure given
by
$$\D(g)=g\otimes g,\ \epsilon(g)=1 \quad \text{and} \quad \D(x)=x\otimes g^t+1\otimes x,\
\epsilon(x)=0,$$ and with $$S(g)=g^{-1}\quad \text{and} \quad
S(x)=-xg^{-t}.$$

As computed in \cite[Subsection 3.3]{BZ}, we have $$\int_H^l \cong
H/\langle x, g-\xi^{-1}\rangle,$$ and the corresponding homomorphism
$\pi$ yields left and right winding automorphisms
\[{\Xi_{\pi}^l:}
\begin{cases}
x\longmapsto x, &\\
g\longmapsto \xi^{-1}g, &
\end{cases} \textrm{and} \;\;\;\;\;
\Xi_{\pi}^r:
\begin{cases}
x\longmapsto \xi^{-t}x, &\\
g\longmapsto \xi^{-1}g. &
\end{cases}\]
So that $G_\pi^l=\langle \Xi_{\pi}^l\rangle$ and $G_\pi^r=\langle
\Xi_{\pi}^r\rangle$ have order $n$. If gcd$(n, t)=1$, then
$G_\pi^l\cap G_\pi^r=\{1\}$ and \cite[Propositon 3.3]{BZ} implies
that there exists a primitive $n$th root $\eta$ of 1 such that
$H(n, t, \xi)\cong H(n, 1, \eta)$ as Hopf algebras. If gcd$(n,
t)\neq 1$, let $m:=n/\text{gcd}(n, t)$, then $G_\pi^l\cap
G_\pi^r=\langle (\Xi_{\pi}^l)^m\rangle$.

Thus we have $\io(H(n, t, \xi))=n$ and $\im(H(n, t, \xi))=m$ for any
$t$. In particular, $\im(H(n, 0, \xi))=1$, $\im(H(n, 1, \xi))=n$ and
$\im(H(n, t, \xi))=m=n/t$ when $t|n$.

Now assume $t|n$ and $m=n/t$, and let $H:=H(n, t, \xi)$, we
calculate the homogeneous parts $H_i^l,\; H_j^r$ and $H_{ij}$ for
our later arguments.
 The gradings start from fixing a primitive $n$th
root $\zeta$ of $1$. We choose $\zeta=\xi^{-1}$. By the expressions
of $\Xi_{\pi}^l$ and  $\Xi_{\pi}^r$, it is not difficult to find
that \begin{equation} H_i^l=k[x]g^i\quad \text{and}\quad
H_j^r=k[xg^{-t}]g^j\end{equation}
 for
all $0\leqslant i, j\leqslant n-1$. Thus we have
\begin{equation}H_{00}=k[x^m]\quad \text{and}\quad H_{i, i+jt}=k[x^m]x^jg^i\end{equation}
for all $0\leqslant i\leqslant n-1, 0\leqslant j\leqslant m-1$.
Moreover we can see that $$H_{ij}=0\;\; \textrm{if}
\;\;i-j\not\equiv 0\; (\textrm{mod}\; t)$$ for all $0\leqslant i,
j\leqslant n-1$.

\noindent$\bullet$ \emph{Generalized Liu algebras}. Let $n$ and
$\omega$ be positive integers. The generalized Liu algebra, denoted
by $B(n, \omega, \gamma)$, is generated by $x^{\pm 1}, g$ and $y$,
 subject to the relations
\begin{equation*}
\begin{cases}
xx^{-1}=x^{-1}x=1,\quad  xg=gx,\quad  xy=yx, & \\
yg=\gamma gy, & \\
y^n=1-x^\omega=1-g^n, & \\
\end{cases}
\end{equation*}
where  $\gamma$ is a primitive $n$th root of 1. The
comultiplication, counit and antipode of $B(n, \omega, \gamma)$  are
given by
$$\Delta(x)=x\otimes x,\quad  \Delta(g)=g\otimes g, \quad   \Delta(y)=y\otimes g+1\otimes y,$$
$$\epsilon(x)=1,\quad  \epsilon(g)=1,\quad    \epsilon(y)=0,$$ and
$$S(x)=x^{-1},\quad  S(g)=g^{-1} \quad S(x)=-yg^{-1}.$$

Let $B:=B(n, \omega, \gamma)$. Using \cite[Lemma 2.6]{LWZ}, we get
$\int_B^l=B/\langle y, x-1, g-\gamma^{-1}\rangle$. The corresponding
homomorphism $\pi$ yields left and right winding automorphisms
\[{\Xi_{\pi}^l:}
\begin{cases}
x\longmapsto x, &\\
g\longmapsto \gamma^{-1}g, &\\
y\longmapsto y, &
\end{cases}
\textrm{and}\;\;\;\; \Xi_{\pi}^r:
\begin{cases}
x\longmapsto x, &\\
g\longmapsto \gamma^{-1}g, &\\
y\longmapsto \gamma^{-1}y. &
\end{cases}\]
Clearly these automorphisms have order $n$ and $G_\pi^l\cap
G_\pi^r=\{1\}$, whence $\io(B)=\im(B)=n$.

Choosing $\zeta=\gamma^{-1}$ for defining gradings of $B$, we can
see
\begin{equation}B_i^l=k\langle x^{\pm 1}, y\rangle g^i \quad \text{and}\quad B_j^r=k\langle x^{\pm 1}, yg^{-1}\rangle g^j\end{equation}
for all $0\leqslant i, j\leqslant n-1$. Thus
\begin{equation}B_{0}=k[x^{\pm
1}]\quad \text{and}\quad B_{ij}=k[x^{\pm
1}]y^{j-i}g^i,\end{equation} where $j-i$ is interpreted mod $n$.

Now, the main results of \cite{BZ} can be formulated in the
following form (for details, see  \cite[Proposition 3.1, Theorems 4.1 and
6.1]{BZ}).

\begin{theorem}\label{t2.7} Assume $H$ is an affine prime regular Hopf algebra of GK-dimensional
one.

\emph{(a)} If $\io(H)=\im(H)=1$, then $H\cong k[x]$ or $H\cong
k[x^{\pm 1}]$;

\emph{(b)} If $\io(H)=n>1$, $\im(H)=1$, then $H\cong H(n, 0, \xi)$
or $H\cong k\mathbb{D}$;

\emph{(c)} If $\io(H)=\im(H)=n$, then $H\cong H(n, 1, \xi)$ or
$H\cong B(n, \omega, \gamma)$.
\end{theorem}

In \cite[Theorem 6.1]{BZ}, the case (c) in the above theorem is
expressed in a more general and convenient form. For our purpose, we
state the general form as follows.

\begin{lemma}\label{l2.8} Let $H$ be an affine prime regular Hopf algebra of GK-dimension
one.  Assume that there exists an algebra homomorphism $\mu: H\to k$
such that $n:=\emph{PI-deg}(H)=|G_\mu^l|$ and $G_\mu^l\cap
G_\mu^r=\{1\}$, where $G_\mu^l$ and $G_\mu^r$ are the groups of left
and right winding automorphisms associated to $\mu$. Then $H$ is
isomorphic as a Hopf algebra either to the Taft algebra $H(n, 1,
\xi)$ or to the generalized Liu algebra $B(n, \omega, \gamma)$. As a
consequence, $\io(H)=\im(H)$.\end{lemma}

\section{Some combinatorial equations}

We collect some combinatorial equations in this section. These equations turn out to be important
for the following analysis.

Throughout this section, $m,d$ are two natural numbers, $\gamma$ is an  $m$th primitive root of $1$ and $\xi$ an element in $k$ satisfying $\xi^{m}=-1$. For each $i\in \mathbbm{Z}$, $\phi_{i}$ is a polynomial defined by
$$\phi_{i}:=1-\gamma^{-i-1}x^{d}.$$
Take $t$ to be an arbitrary integer, define $\bar{t}$ to be the unique element in $\{0,1,\ldots,m-1\}$
satisfying $\bar{t}\equiv t$ {(mod}\,$m$). Then we have $$\phi_{t}=\phi_{\bar{t}}$$
since $\gamma^m=1$.

With this observation, we can use
$${]s,t[}$$
to denote the resulted polynomial by omitting all items
\emph{from} $\phi_{\overline{s}}$ \emph{to} $\phi_{\overline{t}}$ in $\phi_{0}\phi_{1}\cdots \phi_{m-1}$, that is

\begin{equation}\label{eqomit} {]s,t[}=\begin{cases}
\phi_{\bar{t}+1}\cdots \phi_{m-1}\phi_0\cdots\phi_{\bar{s}-1}, & \textrm{if}\; \bar{t}\geqslant \bar{s}
\\1, & \textrm{if}\; \bar{s}=\overline{t}+1 \\
\phi_{\bar{t}+1}\cdots \phi_{\bar{s}-1}, & \textrm{if}\;
\overline{s}\geqslant \bar{t}+2.
\end{cases} \end{equation}
For example, ${]{-1},{-1}[}={]{m-1},{m-1}[}=\phi_0\phi_1\cdots
\phi_{m-2}$.

To study equations with omitting items, the following formula is useful
for us.

\begin{lemma}\cite[Proposition IV.2.7.]{Kas}\label{kas}
Fix an invertible element $q$ of the field $k$. For any scalar $a$
we have
$$(a-z)(a-qz)\cdots (a-q^{n-1}z)=\sum_{l=0}^{n}(-1)^l\binom{n}{l}_q q^{\tfrac{l(l-1)}{2}}a^{n-l}z^l.$$
\end{lemma}

\begin{lemma}\label{ce1} With notions defined as above, we have $$\sum_{j=0}^{m-1}\gamma^{-j}\;{]{j-1},{j-1}[}\;=mx^{(m-1)d}.$$
\end{lemma}
\begin{proof}By Lemma \ref{kas}, $\phi_0\cdots \phi_{m-1}=(1-x^{dm}).$ Note that
\begin{align*}
\sum_{j=0}^{m-1}\;{]{j-1},{j-1}[}\;
&=\sum_{j=0}^{m-1}1+\gamma^{-j}x^d+\gamma^{-2j}x^{2d}+\cdots+\gamma^{-(m-1)j}x^{(m-1)d}\\
                                                                 &=m.
\end{align*}
So
\begin{align*}
& \quad\sum_{j=0}^{m-1}\;{]{j-1},{j-1}[}\;-\sum_{j=0}^{m-1}\gamma^{-j}x^d\;{]{j-1},{j-1}[}\;\\
&=\sum_{j=0}^{m-1}(1-\gamma^{-j}x^d)\;{]{j-1},{j-1}[}\;\\
&=\sum_{j=0}^{m-1}\phi_0\phi_{1}\cdots\phi_{m-1}\\
&=m(1-x^{md}).
\end{align*}
Therefore, $\sum_{j=0}^{m-1}\gamma^{-j}\;{]{j-1},{j-1}[}\;=mx^{(m-1)d}$.
\end{proof}

If we omit two items, then we have
\begin{lemma}\label{ce2}$$\sum_{j=0}^{m-1}\gamma^{-j}\;{]{j-2},{j-1}[}\;=0.$$
\end{lemma}
\begin{proof} By Lemma \ref{kas},
\begin{align*}
{]{j-2},{j-1}[}&=(1-\gamma^{-j-1}x^d)(1-\gamma^{-(j+1)-1}x^d)\cdots(1-\gamma^{-(m+j-3)-1}x^d)\\
               &=\sum_{l=0}^{m-2}(-1)^l\binom{m-2}{l}_{\gamma^{-1}}\gamma^{-\tfrac{l(l-1)}{2}}(\gamma^{-j-1}x^d)^l\\
               &=\sum_{l=0}^{m-2}(-1)^l\binom{m-2}{l}_{\gamma^{-1}}\gamma^{-\tfrac{l(l+1)}{2}-lj}x^{ld}.
\end{align*}
So, to get the result, it is sufficient to show that
$$\sum_{j=0}^{m-1}\gamma^{-j}\gamma^{-lj}=\sum_{j=0}^{m-1}\gamma^{-(l+1)j}=0$$
for all $0\leqslant l \leqslant m-2$. Since $1\leqslant l+1 \leqslant m-1$, it is clear that $\sum_{j=0}^{m-1}\gamma^{-(l+1)j}=0$. Thus the conclusion is established.
\end{proof}

As a direct consequence of this lemma, we get the following basic observation.
\begin{corollary}\label{ce3}$$\sum_{j=0}^{m-1}\xi^{2j}\gamma^{-2j}\;{]{j-2},{j-1}[}=0$$ if and only if $\xi^2=\gamma$, where $\xi$ is an element in $k$ satisfying $\xi^m=-1$.
\end{corollary}
\begin{proof} Let $\theta=\xi^2\gamma^{-2}$, then $\theta$ is an $m$th root of 1. By the proof of the above lemma, it
is sufficient to verify that $\theta\gamma^{-s}\neq 1$ for all
$0\leqslant s \leqslant m-2$. It follows that $\theta=\gamma^{-1}$,
i.e., $\xi^2=\gamma$.
\end{proof}

\begin{lemma}\label{ce5} Fix $i$ such that $1\leqslant i\leqslant m-1$ and let $1\leqslant i'\leqslant i$. Then
$$\sum_{j=0}^{m-1}\gamma^{-i'j}\;{]{j-1-i},{j-1}[}=0$$.
\end{lemma}
\begin{proof} Using Lemma \ref{kas}, we have
\begin{align*}
&\quad {]{j-1-i},{j-1}[}\\
&=\phi_{j}\cdots\phi_{m-1}\phi_{0}\cdots\phi_{j-i-2}\\
&=(1-\gamma^{-j-1}x^d)(1-\gamma^{-(j+1)-1}x^d)\cdots(1-\gamma^{-(m+j-i-2)-1}x^d)\\
&=\sum_{l=0}^{m-1-i}(-1)^l\binom{m-1-i}{l}_{\gamma^{-1}}\gamma^{-\tfrac{l(l-1)}{2}}(\gamma^{-j-1}x^d)^l\\
&=\sum_{l=0}^{m-1-i}(-1)^l\binom{m-1-i}{l}_{\gamma^{-1}}\gamma^{-\tfrac{l(l+1)}{2}-lj}x^{ld}.
\end{align*}
Then
\begin{align*}
&\quad\sum_{j=0}^{m-1}\gamma^{-i'j}\;{]{j-1-i},{j-1}[}\\
&=\sum_{j=0}^{m-1}\gamma^{-i'j}\sum_{l=0}^{m-1-i}(-1)^l\binom{m-1-i}{l}_{\gamma^{-1}}
\gamma^{-\tfrac{l(l+1)}{2}-lj}x^{ld}\\
&=\sum_{l=0}^{m-1-i}(-1)^l\binom{m-1-i}{l}_{\gamma^{-1}}\gamma^{-\tfrac{l(l+1)}{2}}x^{ld}
\sum_{j=0}^{m-1}\gamma^{-(i'+l)j}.
\end{align*}
Since $0\leqslant l \leqslant m-1-i$, $1 \leqslant i'\leqslant i'+l \leqslant
m-1-i+i' \leqslant m-1$. Then we have $\sum_{j=0}^{m-1}\gamma^{-(i'+l)j}=0$ for all
$0\leqslant l \leqslant m-1-i$ and $1\leqslant i'\leqslant i$. This ends the proof.
\end{proof}

The next technical result is also needed.
\begin{lemma}\label{ce7} Let $0\leqslant t \leqslant
i+j\leqslant m-1$,  $0\leqslant l \leqslant m-1-i-j$ and let $q$ be a
primitive $m$th root of 1. Then
\begin{align*}
&\quad q^{\frac{(l+t)(l+t+1)}{2}+t(i+j-t)}\cdot(-1)^{l+t}\binom{m-1-t}{l}_q\binom{m-1+t-i-j}{l+t}_q \\
&=\binom{i+j}{t}_q\binom{m-1-i-j}{l}_q.
\end{align*}
\end{lemma}
\begin{proof} Since
\begin{align*}
\binom{m-1-t}{l}_q\binom{m-1+t-i-j}{l+t}_q
=\frac{(m-1-t)!_q}{(l)!_q(m-1-t-l)!_q}\cdot
\frac{(m-1+t-i-j)!_q}{(l+t)!_q(m-1-l-i-j)!_q}
\end{align*} and
\begin{align*}
\binom{i+j}{t}_q\binom{m-1-i-j}{l}_q
=\frac{(i+j)!_q}{(t)!_q(i+j-t)!_q}\cdot \frac{(m-1-i-j)!_q}{(l)!_q
(m-1-l-i-j)!_q},
\end{align*} we have
\begin{align*}
&\quad \frac{\binom{m-1-t}{l}_q\binom{m-1+t-i-j}{l+t}_q}{
\binom{i+j}{t}_q\binom{m-1-i-j}{l}_q}\\
&=\frac{(m-l-t)_q(m-l+1-t)_q\cdots (m-1-t)_q}{(t+1)_q(t+2)_q\cdots
(t+l)_q} \cdot \frac{ (m-i-j)_q(m-i-j+1)_q\cdots (m-i-j+t-1)_q
}{(i+j-t+1)_q(i+j-t+2)_q\cdots (i+j)_q }.
\end{align*}
Note that for every number $c$, $0\leqslant c \leqslant m-1$,
\begin{align*}
(m-c)_q&=1+q+\cdots +q^{m-1-c}=-(q^{m-c}+q^{m-c+1}\cdots
+q^{m-1})=-q^{m-c}(1+q+\cdots +q^{c-1})\\
&=-q^{m-c}(c)_q.
\end{align*}
Thus
\begin{align*}
\frac{(m-l-t)_q(m-l+1-t)_q\cdots (m-1-t)_q}{(t+1)_q(t+2)_q\cdots
(t+l)_q}=(-1)^lq^{-\tfrac{l(1+l+2t)}{2}}
\end{align*} and
\begin{align*}
\frac{ (m-i-j)_q(m-i-j+1)_q\cdots (m-i-j+t-1)_q
}{(i+j-t+1)_q(i+j-t+2)_q\cdots (i+j)_q
}=(-1)^tq^{-\tfrac{t(1-t+2(i+j))}{2}}.
\end{align*}
Therefore,
\begin{align*}
\frac{\binom{m-1-t}{l}_q\binom{m-1+t-i-j}{l+t}_q}{
\binom{i+j}{t}_q\binom{m-1-i-j}{l}_q}=(-1)^{l+t}
q^{-\tfrac{l(1+l+2t)+t(1-t+2(i+j))}{2}}=(-1)^{l+t}q^{-\frac{(l+t)(l+t+1)}{2}-t(i+j-t)}.
\end{align*}
This completes the proof.
\end{proof}

\section{New Examples}
In this section, we will introduce a new class of algebras $D(m,d,\xi)$ and show that these algebras are
prime regular Hopf algebras of GK-dimension one. Some properties about $D(m,d,\xi)$ are established, in particular, we will show that $D(m,d,\xi)$ is not a pointed Hopf algebra.

\subsection{}{\bf Definition of the Hopf algebra $D(m, d, \xi)$}.

As before, let $m,d$ be two natural numbers satisfying that $(1+m)d$ is even and $\xi$ a primitive $2m$th root of $1$. Define
$$\omega:=md,\;\;\;\;\gamma:=\xi^2.$$

\noindent $\bullet$ \emph{The algebra structure.}  As an algebra, $D(m, d, \xi)$ is generated by $x^{\pm 1}, g^{\pm 1}, y, u_0, u_1, \cdots,
u_{m-1}$, subject to the following relations
\begin{equation}xx^{-1}=x^{-1}x=1,\ gg^{-1}=g^{-1}g=1,\ xg=gx,\ xy=yx,\ yg=\gamma gy,\ y^m=1-x^\omega=1-g^m,\end{equation}
\begin{equation}xu_i=u_ix^{-1},\ yu_i=\phi_iu_{i+1}=\xi x^d u_i y,\ u_i g=\gamma^i x^{-2d}gu_i,\end{equation}
\begin{equation}\label{eq4.3}u_iu_j=\left \{
\begin{array}{lll} (-1)^{-j}\xi^{-j}\gamma^{\frac{j(j+1)}{2}}\frac{1}{m}x^{-\frac{1+m}{2}d}\phi_i\phi_{i+1}\cdots \phi_{m-2-j}y^{i+j}g, & \;\;\text{if}\ \ i+j \leqslant m-2,\\
(-1)^{-j}\xi^{-j}\gamma^{\frac{j(j+1)}{2}}\frac{1}{m}x^{-\frac{1+m}{2}d}y^{i+j}g, & \;\;\text{if}\ \ i+j=m-1,\\
(-1)^{-j}\xi^{-j}\gamma^{\frac{j(j+1)}{2}}\frac{1}{m}x^{-\frac{1+m}{2}d}\phi_i \cdots \phi_{m-1}\phi_0\cdots
\phi_{m-2-j}y^{i+j-m}g, &
\;\;\textrm{otherwise},
\end{array}\right.\end{equation}
where $\phi_i=1-\gamma^{-i-1}x^d$ and $0 \leqslant i, j \leqslant m-1$.

Since $\gamma=\xi^2$, the expression $(-1)^{-j}\xi^{-j}\gamma^{\frac{j(j+1)}{2}}$ in Equation \eqref{eq4.3} equals $(-1)^{-j}\xi^{j^2}$. We still use this expression $(-1)^{-j}\xi^{-j}\gamma^{\frac{j(j+1)}{2}}$ because it is convenient for the further computations involving coproduct and antipode.

To give a unified expression for the last terrible relation \eqref{eq4.3},
we have the following observations.
On one hand, as observed at the beginning of Section 3, if we still define $\phi_{t}=1-\gamma^{-t-1}x^d$ for any $t\in \mathbbm{Z}$,
then $$\phi_{t}=\phi_{\bar{t}},$$
where $\bar{t}\equiv t$ (mod $m$). For any $i,j\in \mathbbm{Z}$, we have
\noindent \[{]{-1-j},{i-1}[}=\begin{cases}
\phi_{\bar{i}}\cdots \phi_{m-1}\phi_0\cdots\phi_{m-2-\bar{j}}, & \textrm{if}\; \bar{i}+ \bar{j}\geqslant m
\\1, & \textrm{if}\; \bar{i}+\bar{j}=m-1 \\
\phi_{\bar{i}}\cdots \phi_{m-2-\bar{j}}, & \textrm{if}\;
\bar{i}+\bar{j}\leqslant m-2.
\end{cases} \]  by \eqref{eqomit}.

For the convenience of our later computations, the next notion is also useful for us,

\noindent \[{[s,t]}:=\begin{cases}
\phi_{\bar{s}}\phi_{\bar{s}+1}\cdots \phi_{\bar{t}}, & \textrm{if}\; \bar{t}\geqslant \bar{s}
\\1, & \textrm{if}\; \bar{s}=\overline{t}+1 \\
\phi_{\bar{s}}\cdots \phi_{m-1}\phi_{0}\cdots \phi_{\bar{t}}, & \textrm{if}\;
\overline{s}\geqslant \bar{t}+2.
\end{cases} \]
In fact, ${[s,t]}$ can be considered as the resulted polynomial (except the case $\bar{s}=\bar{t}+1$) by preserving all items \emph{from} $\phi_{\overline{s}}$
\emph{to} $\phi_{\overline{t}}$ in $\phi_{0}\phi_{1}\cdots \phi_{m-1}$.
So, by definition, we have
\begin{equation}\label{eqpol}
{[i, m-2-j]}={]{-1-j},{i-1}[}.
\end{equation}

On the other hand, we find that
\begin{equation}\label{eq4.4}(-1)^{-km-j}\xi^{-km-j}\gamma^{\tfrac{(km+j)(km+j+1)}{2}}=(-1)^{-j}\xi^{-j}\gamma^{\tfrac{j(j+1)}{2}}\end{equation}
for any $k\in \mathbbm{Z}$. Therefore, if we define
$$u_{s}:=u_{\overline{s}},$$
then the relation \eqref{eq4.3} can be replaced by
\begin{align}
u_iu_j&=(-1)^{-j}\xi^{-j}\gamma^{\frac{j(j+1)}{2}}\frac{1}{m}x^{-\tfrac{1+m}{2}d}\;{]{-1-j},{i-1}[}\;y^{\overline{i+j}}g\\
   \notag  &=(-1)^{-j}\xi^{-j}\gamma^{\frac{j(j+1)}{2}}\frac{1}{m}x^{-\tfrac{1+m}{2}d}\;{[i, m-2-j]}\;y^{\overline{i+j}}g
\end{align} for all $i, j\in \mathbbm{Z}$.

We give a bigrading on this algebra for use later. Define the
following two algebra automorphisms of $D(m,d,\xi)$:
\[{\Xi_{\pi}^l:}
\begin{cases}
x\longmapsto x, &\\
y\longmapsto y, &\\
g\longmapsto \gamma^{-1}g, &\\
u_i\longmapsto \xi^{-1}u_i, &
\end{cases}
\textrm{and}\;\;\;\; \Xi_{\pi}^r:
\begin{cases}
x\longmapsto x, &\\
y\longmapsto \gamma^{-1}y, &\\
g\longmapsto \gamma^{-1}g, &\\
u_i\longmapsto \xi^{-(2i+1)} u_i.&
\end{cases}\]
It is straightforward to show that $\Xi_{\pi}^l$ and $\Xi_{\pi}^r$ are indeed
algebra automorphisms of $D(m,d,\xi)$ and these automorphisms have order $2m$ by noting that $\xi$ is a primitive $2m$th root of 1 and $u_i\neq 0$ in $D(m,d,\xi)$ for all $i$ (if one $u_i=0$ in $D(m,d,\xi)$, then $y^{\overline{i+j}}g=0$ by \eqref{eq4.3}, which is absurd).
Choosing $\zeta=\xi^{-1}$,
define
$$D_i^l:=\{h\in D(m,d,\xi)|\Xi_{\pi}^l(h)=\zeta^{i}h\},\;\;D_j^r:=\{h\in D(m,d,\xi)|\Xi_{\pi}^r(h)=\zeta^{j}h\}$$
for $0\leqslant i,j\leqslant 2m-1$. Direct computations show that
\[{D_i^l=}
\begin{cases}
k\langle x^{\pm 1}, y\rangle g^{\tfrac{i}{2}}, & i=\textrm{even},\\
\sum_{s=0}^{m-1}k[ x^{\pm 1}] g^{\tfrac{i-1}{2}}u_s,
& i=\textrm{odd},
\end{cases}\]
and \[ D_j^r=
\begin{cases}
k\langle x^{\pm 1}, yg^{-1}\rangle g^{\tfrac{j}{2}}, & j=\textrm{even},\\
\sum_{s=0}^{m-1}k[x^{\pm 1}]
g^su_{\tfrac{j-1}{2}-s}, & j=\textrm{odd}.
\end{cases}\]
Therefore \begin{equation}\label{eqDij} D_{ij}:=D_i^l\cap D_j^r=
\begin{cases}
k[x^{\pm 1}]y^{\frac{j-i}{2}}g^{\frac{i}{2}}, & i, j=\textrm{even},\\
k[x^{\pm 1}]g^{\tfrac{i-1}{2}}u_{\tfrac{j-i}{2}}, & i, j=\textrm{odd},\\
0, & \textrm{otherwise}.
\end{cases}\end{equation}
Since $\sum_{i,j}D_{ij}=D(m,d,\xi)$, we have
\begin{equation}\label{eq4.6}D(m,d,\xi)=\bigoplus_{i,j=0}^{2m-1}D_{ij}\end{equation}
 which is a bigrading on $D(m,d,\xi)$ automatically.

 Let $D:= D(m,d,\xi)$, then $D\otimes D$ is graded naturally by inheriting the grading defined above. In particular,
 for any $h\in D\otimes D$,  we use
 $$h_{(s_1,t_1)\otimes (s_2,t_2)}$$
 to denote the homogeneous part of $h$ in $D_{s_{1},t_{1}}\otimes D_{s_{2},t_{2}} $. This notion will be used freely
 in the proof of  Proposition \ref{p4.2}.

\noindent $\bullet$ \emph{The coalgebra structure and the antipode.} The coproduct $\D$, the counit $\epsilon$ and the antipode $S$ of $D(m,d,\xi)$ are given by
$$\D(x)=x\otimes x,\;\; \D(g)=g\otimes g, \;\;\D(y)=y\otimes g+1\otimes y,$$
$$\D(u_i)=\sum_{j=0}^{m-1}\gamma^{j(i-j)}u_j\otimes x^{-jd}g^ju_{i-j};$$
$$\epsilon(x)=\epsilon(g)=\epsilon(u_0)=1,\;\;\epsilon(y)=\epsilon(u_s)=0;$$
$$S(x)=x^{-1},\;\; S(g)=g^{-1}, \;\;S(y)=-yg^{-1},$$
\begin{equation}\label{eq4.8}S(u_i)=(-1)^i\xi^{-i}\gamma^{-\frac{i(i+1)}{2}}x^{id+\frac{3}{2}(1-m)d}g^{m-i-1}u_i,\end{equation}
for $0\leq i\leq m-1$ and $1\leqslant s\leqslant m-1$.

Since $g^m=x^{md}$ and \eqref{eq4.4} , the definition about $S(u_{i})$ still holds for any integer $i$, that is, \eqref{eq4.8}
can be replaced in a more convenient way:
\begin{equation}S(u_s)=(-1)^s\xi^{-s}\gamma^{-\frac{s(s+1)}{2}}x^{sd+\frac{3}{2}(1-m)d}g^{m-s-1}u_s
\end{equation}
for all $s\in \mathbbm{Z}$.

\begin{remark}\label{r4.1}  Recall that $\xi^2=\gamma$. It is not hard to see that the subalgebra of $D(m, d, \xi)$ generated by $x^{\pm 1}, g^{\pm 1}, y$ is exact the generalized Liu algebra $B(m, \omega, \gamma)$. Indeed, by Equations \eqref{eqDij} and \eqref{eq4.6}, $B_{ij}=D_{2i, 2j}$ for all $0\leqslant i, j\leqslant m-1$. By definition, $D(m, d, \xi)$ is affine. Moreover, $D(m, d, \xi)$ is a finitely generated $k[x^{\pm 1}]$-module by Equation \eqref{eqDij}. Thus $D(m, d, \xi)$ has GK-dimension one. At the same time, note that $Z=k[z_s|z_s:=x^s+x^{-s}, s\in \mathbb{N}_0]$ lies in the center of $D(m, d, \xi)$, which implies that $D(m, d, \xi)$ is PI. In one word, $D(m, d, \xi)$ is affine, PI and has GK-dimension one, and contains $B(m, \omega, \gamma)$ as a Hopf subalgebra.
\end{remark}

\subsection{}{\bf $D(m, d, \xi)$ is a Hopf algebra}.

The main aim of this subsection is to show that $D(m,d,\xi)$ is indeed a Hopf algebra.

\begin{proposition} \label{p4.2}The algebra $D(m,d,\xi)$ defined above is a Hopf algebra.
\end{proposition}

\noindent\emph{Proof.} The proof is standard but not easy. For completeness and the convenience of the reader,
we give the proof here.  As usual, we decompose the proof into several steps.  Since the subalgebra generated by
$x^{\pm 1}, y, g$ is just the generalized Liu algebra $B(m, \omega, \gamma)$, which is a Hopf algebra already, we only need to verify the related relations in $D(m,d,\xi)$ where $u_i$ are involved.

\noindent $\bullet$ \emph{Step} 1 ($\D$ and $\epsilon$ are algebra
homomorphisms).

First of all, it is clear that $\epsilon$ is an algebra
homomorphism. Since $x$ and $g$ are group-like elements, the
verifications of $\D(x)\D(u_i)=\D(u_i)\D(x^{-1})$ and
$\D(u_i)\D(g)=\gamma^{i}\D(x^{-2d})\D(g)\D(u_i)$ are simple and so they are omitted.

\noindent (1) \emph{The proof of} $\D(\phi_i)\D(u_{i+1})=\D(y)\D(u_i)=\xi\D(x^d)\D(u_i)\D(y)$.

By definition $\D(u_i)=\sum_{j=0}^{m-1}\gamma^{j(i-j)}u_j\otimes
x^{-jd}g^ju_{i-j}$ for all $0 \leqslant i \leqslant m-1$, we have
\begin{align*}
\D(\phi_i)\D(u_{i+1})&=(1\otimes 1 - \gamma^{-i-1}x^d \otimes x^d)\sum_{j=0}^{m-1}\gamma^{j(i+1-j)}u_j\otimes x^{-jd}g^ju_{i+1-j}\\
                 &=\sum_{r=0}^{m-1}\gamma^{r(i+1-r)}u_r\otimes x^{-rd}g^ru_{i+1-r}-\sum_{l=0}^{m-1}\gamma^{l(i+1-l)-i-1}x^du_l\otimes x^{(1-l)d}g^ju_{i+1-l}.
\end{align*}
And
\begin{align*}
\D(y)\D(u_i)&=(y\otimes g +1\otimes y)\sum_{j=0}^{m-1}\gamma^{j(i-j)}u_j\otimes x^{-jd}g^ju_{i-j}\\
            &=\sum_{l=0}^{m-1}\gamma^{l(i-l)}yu_l\otimes x^{-ld}g^{l+1}u_{i-l}+ \sum_{r=0}^{m-1}\gamma^{r(i-r)}u_r\otimes yx^{-rd}g^ru_{i-r}\\
            &=\sum_{l=0}^{m-1}\gamma^{l(i-l)}u_{l+1}\otimes x^{-ld}g^{l+1}u_{i-l}- \sum_{l=0}^{m-1}\gamma^{l(i-l)-l-1}x^du_{l+1}\otimes
            x^{-ld}g^{l+1}u_{i-l}\\
            &\quad +\sum_{r=0}^{m-1}\gamma^{r(i-r+1)}u_{r}\otimes x^{-rd}g^{r}u_{i+1-r}- \sum_{r=0}^{m-1}\gamma^{(r-1)(i+1-r)}u_{r}\otimes
            x^{(1-r)d}g^{r}u_{i+1-r}\\
            &=\sum_{r=0}^{m-1}\gamma^{r(i-r+1)}u_{r}\otimes x^{-rd}g^{r}u_{i+1-r}- \sum_{l=0}^{m-1}\gamma^{l(i-l)-l-1}x^du_{l+1}\otimes
            x^{-ld}g^{l+1}u_{i-l}.\end{align*}
Hence $\D(\phi_i)\D(u_{i+1})=\D(y)\D(u_i)$. Similarly,
\begin{align*}
\xi\D(x^d)\D(u_i)\D(y)&=\xi\cdot(x^d\otimes x^d)(\sum_{j=0}^{m-1}\gamma^{j(i-j)}u_j\otimes x^{-jd}g^ju_{i-j})(y\otimes g +1\otimes y)\\
            &=\sum_{s=0}^{m-1}\gamma^{s(i-s)}yu_s\otimes x^{(1-s)d}g^{s}u_{i-s}g+ \sum_{t=0}^{m-1}\gamma^{t(i-t)}x^du_t\otimes x^{-td}g^tyu_{i-t}\\
            &=\sum_{s=0}^{m-1}\gamma^{(s+1)(i-s)}u_{s+1}\otimes x^{-(s+1)d}g^{s+1}u_{i-s}- \sum_{s=0}^{m-1}\gamma^{(s+1)(i-s-1)}x^du_{s+1}\otimes x^{-(s+1)d}g^{s+1}u_{i-s}\\
            &\quad +\sum_{t=0}^{m-1}\gamma^{t(i-t)}x^du_t\otimes x^{-td}g^tu_{i+1-t}- \sum_{t=0}^{m-1}\gamma^{(t-1)(i-t)-1}x^du_t\otimes x^{(1-t)d}g^tu_{i+1-t}\\
            &=\sum_{s=0}^{m-1}\gamma^{(s+1)(i-s)}u_{s+1}\otimes x^{-(s+1)d}g^{s+1}u_{i-s}- \sum_{t=0}^{m-1}\gamma^{(t-1)(i-t)-1}x^du_t\otimes x^{(1-t)d}g^tu_{i+1-t},\end{align*}
   which equals $\D(\phi_i)\D(u_{i+1})$ clearly.

\noindent (2)\emph{ The proof of} $\D(u_iu_j)=\D(u_i)\D(u_j)$.

We have that
\begin{align*}
\D(u_i)\D(u_j)&=\sum_{s=0}^{m-1}\gamma^{s(i-s)}u_s\otimes x^{-sd}g^su_{i-s}\cdot \sum_{t=0}^{m-1}\gamma^{t(j-t)}u_t\otimes x^{-td}g^tu_{j-t}\\
              &=\sum_{t=0}^{m-1}\sum_{s=0}^{m-1}\gamma^{s(i-s)}u_s\gamma^{(t-s)(j-t+s)}u_{t-s}\otimes x^{-sd}g^su_{i-s}x^{-(t-s)d}g^{t-s}u_{j-t+s}\\
              &=\sum_{t=0}^{m-1}\sum_{s=0}^{m-1}\gamma^{(t-s)(j-t+s)+(i-s)t}u_su_{t-s}\otimes x^{-td}g^tu_{i-s}u_{j-t+s}.
\end{align*}
By the bigrading given in \eqref{eq4.6}, we can find that for each $0\leqslant t \leqslant m-1$,
$$\sum_{s=0}^{m-1}\gamma^{(t-s)(j-t+s)+(i-s)t}u_su_{t-s}\otimes
x^{-td}g^tu_{i-s}u_{j-t+s}\in D_{2, 2+2t}\otimes D_{2+2t,
2+2(i+j)},$$ where the suffixes in $ D_{2, 2+2t}\otimes D_{2+2t,
2+2(i+j)}$ are interpreted mod $2m$.

Note that
$$u_su_{t-s}=(-1)^{-(t-s)}\xi^{-(t-s)}\gamma^{\frac{(t-s)(t-s+1)}{2}}\frac{1}{m}x^{-\frac{1+m}{2}d}\;{[s, m-2-t+s]}\; y^{t}g$$
and
$$u_{i-s}u_{j-t+s}=(-1)^{-(j-t+s)}\xi^{-(j-t+s)}\gamma^{\frac{(j-t+s)(j-t+s+1)}{2}}\frac{1}{m}x^{-\frac{1+m}{2}d}\;{[i-s, m-2+t-j-s]}\;y^{\overline{i+j-t}}g.$$

Using Lemma \ref{kas}, we get
\begin{align*}
\;{[s, m-2-t+s]}&=(1-\gamma^{-s-1}x^d)(1-\gamma^{-s-2}x^d)\cdots (1-\gamma^{-(m-2-t+s)-1}x^d)\\
           &=\sum_{l=0}^{m-1-t}(-1)^l\binom{m-1-t}{l}_{\gamma^{-1}}\gamma^{-\tfrac{l(l-1)}{2}}(\gamma^{-s-1}x^d)^l\\
           &=\sum_{l=0}^{m-1-t}(-1)^l\binom{m-1-t}{l}_{\gamma^{-1}}
           \gamma^{-\tfrac{l(l+1)}{2}-sl}x^{ld},
\end{align*} and
\begin{align*}
\;{[i-s, m-2+t-j-s]}&=(1-\gamma^{-(i-s)-1}x^d)(1-\gamma^{-(i-s+1)-1}x^d)\cdots (1-\gamma^{-(i-s+m-2-\overline{i+j-t})-1}x^d)\\
           &=\sum_{r=0}^{m-1-\overline{i+j-t}}(-1)^r\binom{m-1-\overline{i+j-t}}{r}_{\gamma^{-1}}\gamma^{-\tfrac{r(r-1)}{2}}(\gamma^{s-i-1}x^d)^r\\
           &=\sum_{r=0}^{m-1-\overline{i+j-t}}(-1)^r\binom{m-1-\overline{i+j-t}}{r}_{\gamma^{-1}}\gamma^{-\tfrac{r(r+1)}{2}+(s-i)r} x^{rd}.
\end{align*}

Then for each $0\leqslant t \leqslant m-1$,
\begin{align*}
&\quad \D(u_i)\D(u_j)_{{(2, 2+2t)}\otimes {(2+2t,2+2(i+j))}}\\
&=\sum_{s=0}^{m-1}\gamma^{(t-s)(j-t+s)+(i-s)t}u_su_{t-s}\otimes x^{-td}g^tu_{i-s}u_{j-t+s}\\
&=\sum_{s=0}^{m-1}\gamma^{(t-s)(j-t+s)+(i-s)t}(-1)^{-(t-s)}\xi^{-(t-s)}\gamma^{\frac{(t-s)(t-s+1)}{2}}\frac{1}{m}x^{-\frac{1+m}{2}d}\;{[s, m-2-t+s]}\;y^{t}g\\
&\quad\quad\otimes x^{-td}g^t(-1)^{-(j-t+s)}\xi^{-(j-t+s)}\gamma^{\frac{(j-t+s)(j-t+s+1)}{2}}\frac{1}{m}x^{-\frac{1+m}{2}d}\;{[i-s, m-2+t-j-s]}\;y^{\overline{i+j-t}}g\\
&=(-1)^{-j}\xi^{-j}\frac{1}{m^2}(\sum_{s=0}^{m-1}\gamma^{\frac{j^2+j}{2}+(i-s)t-t(i+j-t)}\;{[s, m-2-t+s]}\;\otimes x^{-td}\;{[i-s, m-2+t-j-s]}\;)\\
&\quad \quad ( x^{-\frac{1+m}{2}d}y^{t}g\otimes x^{-\frac{1+m}{2}d}y^{\overline{i+j-t}}g^{t+1})\\
&=(-1)^{-j}\xi^{-j}\frac{1}{m^2}\sum_{l=0}^{m-1-t} \sum_{r=0}^{m-1-\overline{i+j-t}}\gamma^{\frac{j(j+1)-l(l+1)-r(r+1)}{2}-t(j-t)-ir} (-1)^{l+r}\binom{m-1-t}{l}_{\gamma^{-1}}\\
&\quad\quad\binom{m-1-\overline{i+j-t}}{r}_{\gamma^{-1}} \sum_{s=0}^{m-1}\gamma^{(r-l-t)s}\cdot (x^{ld} \otimes  x^{(r-t)d}) \cdot( x^{-\frac{1+m}{2}d}y^{t}g\otimes x^{-\frac{1+m}{2}d}y^{\overline{i+j-t}}g^{t+1}).
\end{align*}

Meanwhile,
$u_iu_j=(-1)^{-j}\xi^{-j}\gamma^{\frac{j(j+1)}{2}}\frac{1}{m}x^{-\frac{1+m}{2}d}\;{[i, m-2-j]}\;y^{\overline{i+j}}g$. Since
\begin{align*}
\D(y^{\overline{i+j}})&=(1\otimes y+ y\otimes g)^{\overline{i+j}}\\
           &=\sum_{t=0}^{\overline{i+j}}\binom{\overline{i+j}}{t}_{\gamma^{-1}}(1\otimes y)^{\overline{i+j}-t}\cdot (y\otimes g)^t\\
           &=\sum_{t=0}^{\overline{i+j}}\binom{\overline{i+j}}{t}_{\gamma^{-1}}y^t\otimes y^{\overline{i+j}-t}g^t
\end{align*}
and
\begin{align*}
\D({[i, m-2-j]})&=(1\otimes 1 - \gamma^{-i-1}x^d\otimes x^d)
(1\otimes 1 - \gamma^{-i-2}x^d\otimes x^d)\cdots (1\otimes 1 - \gamma^{-(i+m-2-\overline{i+j})-1}x^d\otimes x^d)\\
           &=\sum_{l=0}^{m-1-\overline{i+j}}(-1)^l\binom{m-1-\overline{i+j}}{l}_{\gamma^{-1}}\gamma^{-\tfrac{l(l-1)}{2}}(\gamma^{-i-1}x^d\otimes x^d)^l\\
           &=\sum_{l=0}^{m-1-\overline{i+j}}(-1)^l\binom{m-1-\overline{i+j}}{l}_{\gamma^{-1}}\gamma^{-\tfrac{l(l+1)}{2}-il}\cdot x^{ld}\otimes x^{ld},
\end{align*}
we get
\begin{align*}
\D(u_iu_j)&=\D((-1)^{-j}\xi^{-j}\gamma^{\frac{j(j+1)}{2}}\frac{1}{m}x^{-\frac{1+m}{2}d}\;{[i, m-2-j]}\;y^{\overline{i+j}}g)\\
          &=(-1)^{-j}\xi^{-j}\gamma^{\frac{j(j+1)}{2}}\frac{1}{m} \D(x^{-\frac{1+m}{2}d})\D(\;{[i, m-2-j]}\;)\D(y^{\overline{i+j}})\D(g)\\
          &=(-1)^{-j}\xi^{-j}\gamma^{\frac{j(j+1)}{2}}\frac{1}{m} \sum_{l=0}^{m-1-\overline{i+j}}(-1)^l\binom{m-1-\overline{i+j}}{l}_{\gamma^{-1}}\gamma^{-\tfrac{l(l+1)}{2}-il} \sum_{t=0}^{\overline{i+j}}\binom{\overline{i+j}}{t}_{\gamma^{-1}} \\
          &\quad\quad (x^{-\frac{1+m}{2}d}\otimes x^{-\frac{1+m}{2}d})\cdot (x^{ld}\otimes x^{ld})\cdot (y^t\otimes y^{\overline{i+j-t}}g^t)\cdot (g\otimes g)\\
          &=(-1)^{-j}\xi^{-j}\gamma^{\frac{j(j+1)}{2}}\frac{1}{m} \sum_{t=0}^{\overline{i+j}}\cdot \sum_{l=0}^{m-1-\overline{i+j}}(-1)^l \binom{\overline{i+j}}{t}_{\gamma^{-1}}\binom{m-1-\overline{i+j}}{l}_{\gamma^{-1}}\gamma^{-\tfrac{l(l+1)}{2}-il} \\
          &\quad\quad (x^{ld}\otimes x^{ld})\cdot (x^{-\frac{1+m}{2}d}y^tg \otimes x^{-\frac{1+m}{2}d}y^{\overline{i+j-t}}g^{t+1}).
\end{align*}
For each $0\leqslant t \leqslant \overline{i+j}$, $(x^{ld}\otimes x^{ld})\cdot
(x^{-\frac{1+m}{2}d}y^tg \otimes
x^{-\frac{1+m}{2}d}y^{\overline{i+j-t}}g^{t+1}) \in D_{2, 2+2t}\otimes D_{2+2t,
2+2(i+j)}$ for any  $l$. So,
\begin{align*}
&\quad\D(u_iu_j)_{{(2, 2+2t)}\otimes {(2+2t,2+2(i+j))}}\\
&=(-1)^{-j}\xi^{-j}\gamma^{\frac{j(j+1)}{2}}\frac{1}{m} \sum_{l=0}^{m-1-\overline{i+j}}(-1)^l \binom{\overline{i+j}}{t}_{\gamma^{-1}}\binom{m-1-\overline{i+j}}{l}_{\gamma^{-1}}\gamma^{-\tfrac{l(l+1)}{2}-il} \\
&\quad\quad (x^{ld}\otimes x^{ld})\cdot
(x^{-\frac{1+m}{2}d}y^tg \otimes
x^{-\frac{1+m}{2}d}y^{\overline{i+j-t}}g^{t+1}).
\end{align*}

By the graded structure of $D\otimes D$,
$\D(u_i)\D(u_j)=\D(u_iu_j)$ if and only if
\begin{equation}\label{4.10}\D(u_i)\D(u_j)_{{(2, 2+2t)}\otimes {(2+2t,2+2(i+j))}}=0\end{equation} for
all $\overline{i+j}+1\leqslant t \leqslant m-1$ and
\begin{equation}\label{4.11}\D(u_iu_j)_{{(2, 2+2t)}\otimes
{(2+2t,2+2(i+j))}}=\D(u_i)\D(u_j)_{{(2, 2+2t)}\otimes
{(2+2t,2+2(i+j))}} \end{equation} for all $0\leqslant t \leqslant
\overline{i+j}$.

By the expression of $\D(u_i)\D(u_j)_{{(2, 2+2t)}\otimes {(2+2t,2+2(i+j))}}$, we can find that
it is zero if $\sum\limits_{s=0}^{m-1}\gamma^{(r-l-t)s}$ $=0$. Note that in the case of $\overline{i+j}+1\leqslant t \leqslant m-1$, $m-1-\overline{i+j-t}=t-1-\overline{i+j}$. So $0\leqslant r\leqslant t-1-\overline{i+j}$ and thus $1-m \leqslant r-l-t \leqslant -1-\overline{i+j}$. This means that in this case we always have
$$\sum_{s=0}^{m-1}\gamma^{(r-l-t)s}=0,$$ which
implies \eqref{4.10}.

Now let $0\leqslant t \leqslant \overline{i+j}$. Then $1-m \leqslant r-l-t \leqslant m-1-\overline{i+j-t}-t<m$.
As discussed above, $\D(u_i)\D(u_j)_{{(2,
2+2t)}\otimes {(2+2t,2+2(i+j))}}=0$ if $r-l-t\neq 0$.  So we only need to verify the case when $r=l+t$. At this time,
$0\leqslant l \leqslant m-1-\overline{i+j}$. Then \eqref{4.11} holds if and
only if
\begin{align*}
&\quad\gamma^{-\frac{(l+t)(l+t+1)}{2}-t(i+j-t)}\cdot(-1)^{l+t}\binom{m-1-t}{l}_{\gamma^{-1}}\binom{m-1-\overline{i+j-t}}{l+t}_{\gamma^{-1}} \\
&=\binom{\overline{i+j}}{t}_{\gamma^{-1}}\binom{m-1-\overline{i+j}}{l}_{\gamma^{-1}},
\end{align*} which is just Lemma \ref{ce7} by setting $q=\gamma^{-1}$.\\

\noindent{$\bullet$ \emph{Step} 2 (Coassociative and couint).

Indeed, for each $0\leqslant i\leqslant m-1$
\begin{align*}
(\D\otimes \Id)\D(u_i)&=(\D\otimes \Id)(\sum_{j=0}^{m-1}\gamma^{j(i-j)}u_j\otimes x^{-jd}g^ju_{i-j})\\
                     &=\sum_{j=0}^{m-1}\gamma^{j(i-j)}(\sum_{s=0}^{m-1}\gamma^{s(j-s)}u_s\otimes x^{-sd}g^su_{j-s})\otimes x^{-jd}g^ju_{i-j}\\
                     &=\sum_{j,s=0}^{m-1}\gamma^{j(i-j)+s(j-s)}u_s\otimes x^{-sd}g^su_{j-s}\otimes x^{-jd}g^ju_{i-j},
\end{align*}
and
\begin{align*}
(\Id\otimes \D)\D(u_i)&=(\Id\otimes \D)(\sum_{s=0}^{m-1}\gamma^{s(i-s)}u_s\otimes x^{-sd}g^su_{i-s})\\
                     &=\sum_{s=0}^{m-1}\gamma^{s(i-s)}
                     u_s\otimes (\sum_{t=0}^{m-1}\gamma^{t(i-s-t)}x^{-sd}g^su_{t}\otimes x^{-sd}g^s x^{-td}g^{t}u_{i-s-t})\\
                     &=\sum_{s,t=0}^{m-1}\gamma^{s(i-s)+t(i-s-t)}u_s\otimes x^{-sd}g^su_{t}\otimes x^{-(s+t)d}g^{(s+t)}u_{i-s-t}.
\end{align*}

It is not hard to see that $(\D\otimes \Id)\D(u_i)=(\Id\otimes
\D)\D(u_i)$ for all $0\leqslant i \leqslant m-1$.  The verification of $(\epsilon\otimes \Id)\D(u_i)=(\Id\otimes
\epsilon)\D(u_i)=u_i$ is easy and it is omitted.\\

\noindent $\bullet$ \emph{Step} 3 (Antipode is an algebra anti-homomorphism).

Because $x$ and $g$ are group-like elements, we only check
$$S(u_{i+1})S(\phi_i)=S(u_i)S(y)=\xi S(y)S(u_i)S(x^d)$$ and
$$S(u_iu_j)=S(u_j)S(u_i)$$ here.

\noindent (1) \emph{The proof of} $S(u_{i+1})S(\phi_i)=S(u_i)S(y)=\xi S(y)S(u_i)S(x^d)$.

Since $u_iS(\phi_j)=\phi_ju_i$ for all $i, j$,
$$S(u_{i+1})S(\phi_i)=(-1)^{i+1}\xi^{-(i+1)}\gamma^{-\frac{(i+1)(i+2)}{2}}x^{(i+1)d+\frac{3}{2}(1-m)d}g^{m-i-2}u_{i+1}S(\phi_i)=\phi_iS(u_{i+1}).$$

Through direct calculation, we have
\begin{align*}
S(u_i)S(y)&=(-1)^i\xi^{-i}\gamma^{-\frac{i(i+1)}{2}}x^{id+
\frac{3}{2}(1-m)d}g^{m-i-1}u_i\cdot (-yg^{-1})\\
         &=(-1)^{i+1}\xi^{-(i+1)}\gamma^{-\frac{i(i+1)}{2}}x^{(i-1)d+
         \frac{3}{2}(1-m)d}g^{m-i-1}yu_{i}g^{-1}\\
         &=\phi_i\cdot(-1)^{i+1}\xi^{-(i+1)}
         \gamma^{-\frac{(i+1)(i+2)}{2}}x^{(i+1)d+\frac{3}{2}(1-m)d}g^{m-i-2}u_{i+1}\\
         &=\phi_iS(u_{i+1}),
\end{align*}
and
\begin{align*}
\xi S(y)S(u_i)S(x^d)&=-\xi y g^{-1}\cdot (-1)^i\xi^{-i}\gamma^{-\frac{i(i+1)}{2}}x^{id+\frac{3}{2}(1-m)d}g^{m-i-1}u_ix^{-d}\\
         &=(-1)^{i+1}\xi^{-(i+1)}\gamma^{-\frac{(i+1)(i+2)}{2}}x^{(i+1)d+\frac{3}{2}(1-m)d}g^{m-i-2}yu_{i}\\
         &=\phi_iS(u_{i+1}).
\end{align*}

\noindent (2) \emph{The proof of}  $S(u_iu_j)=S(u_j)S(u_i)$.

Define $\overline{\phi_s}:=1-\gamma^{-s-1}x^{-d}$ for all $s\in \mathbb{Z}$. Using this notion,
$$x^d\overline{\phi_s}=x^d(1-\gamma^{-s-1}x^{-d})=-\gamma^{-s-1}(1-\gamma^{-(m-s-2)-1}x^d)=-\gamma^{-s-1}\phi_{m-s-2}.$$
And so

\begin{align*}
S(u_iu_j)&=S((-1)^{-j}\xi^{-j}\gamma^{\frac{j(j+1)}{2}}\frac{1}{m}x^{-\frac{1+m}{2}d}\;{[i, m-2-j]}\;y^{\overline{i+j}}g)\\
         &=(-1)^{-j}\xi^{-j}\gamma^{\frac{j(j+1)}{2}}\frac{1}{m}S(g)S(y^{i+j})S({[i, m-2-j]})S(x^{-\frac{1+m}{2}d})\\
         &=(-1)^{-j}\xi^{-j}\gamma^{\frac{j(j+1)}{2}}\frac{1}{m}g^{-1}(-yg^{-1})^{\overline{i+j}}S({[i, m-2-j]})x^{\frac{1+m}{2}d}\\
         &=(-1)^{\overline{i+j}-j}\xi^{-j}\gamma^{\frac{j(j+1)+(\overline{i+j})(\overline{i+j}+1)}{2}}\frac{1}{m}x^{\frac{1+m}{2}d}S({[i, m-2-j]})y^{\overline{i+j}}g^{-(\overline{i+j}+1)}\\
         &=(-1)^{m-1-j}\xi^{-j}\gamma^{\frac{j(j+1)+(\overline{i+j})(\overline{i+j}+1)+(m-1-\overline{i+j})(-m-2i+\overline{i+j})}{2}}\frac{1}{m}x^{\frac{1+m}{2}d-(m-1-\overline{i+j})d}\\ &\quad \quad {[j, m-2-i]}
         y^{\overline{i+j}}g^{-(\overline{i+j}+1)}\\
         &=(-1)^{-j}\xi^{-j}\gamma^{\frac{j^2+j}{2}+i(\overline{i+j}+1)}
         \frac{1}{m}x^{\frac{1+m}{2}d-(m-1-\overline{i+j})d}\;{[j, m-2-i]}\;y^{\overline{i+j}}g^{-(\overline{i+j}+1)}\\
         &=(-1)^{j}\xi^{-j}\gamma^{\frac{j^2+j}{2}+i(i+j+1)}\frac{1}{m}x^{\frac{3-m}{2}d+(i+j)d}\;{[j, m-2-i]}\;y^{\overline{i+j}}g^{-(i+j+1)},
\end{align*}

\begin{align*}
S(u_j)S(u_i)&=(-1)^j\xi^{-j}\gamma^{-\frac{j(j+1)}{2}}x^{jd+\frac{3}{2}(1-m)d}g^{m-j-1}u_j\cdot (-1)^i\xi^{-i}\gamma^{-\frac{i(i+1)}{2}}x^{id+\frac{3}{2}(1-m)d}g^{m-i-1}u_i\\
         &=(-1)^{i+j}\xi^{-i-j}\gamma^{-\frac{i(i+1)+j(j+1)}{2}}x^{(j-i)d}g^{m-j-1}u_jg^{m-i-1}u_i\\
         &=(-1)^{i+j}\xi^{-i-j}\gamma^{-\frac{i(i+1)+j(j+1)}{2}-j(i+1)}x^{(i+j+2)d}g^{-(i+j+2)}u_ju_i\\
         &=(-1)^{j}\xi^{-2i-j}\gamma^{-\frac{j(j+1)}{2}-j(i+1)+(i+j)(i+j+2)}\frac{1}{m}x^{-\frac{1+m}{2}d+(i+j+2)d}\;{[i, m-2-j]}\;y^{\overline{i+j}}g^{-(i+j+1)}\\
         &=(-1)^{j}\xi^{-j}\gamma^{\frac{j^2+j}{2}+i(i+j+1)}\frac{1}{m}x^{\frac{3-m}{2}d+(i+j)d}\;{[i, m-2-j]}\;y^{\overline{i+j}}g^{-(i+j+1)}.
\end{align*}
The proof is done.\\

\noindent $\bullet$ \emph{Step} 4
($(S*\Id)(u_i)=(\Id*S)(u_i)=\epsilon(u_i)$).

In fact,

\begin{align*}
(S*\Id)(u_0)&=\sum_{j=0}^{m-1}S(\gamma^{-j^2}u_j)x^{-jd}g^ju_{-j} \\
           &=\sum_{j=0}^{m-1}\gamma^{-j^2} (-1)^j \xi^{-j} \gamma^{-\frac{j(j+1)}{2}}x^{jd+\frac{3}{2}(1-m)d} g^{m-j-1} u_jx^{-jd}g^ju_{-j} \\
           &=x^{\frac{3}{2}(1-m)d}g^{m-1}(\sum_{j=0}^{m-1}(-1)^j \xi^{-j}\gamma^{-\frac{j(j+1)}{2}}u_ju_{-j})\\
           &=x^{\frac{3}{2}(1-m)d}g^{m-1}(\sum_{j=0}^{m-1}\gamma^{-j}
             \frac{1}{m} x^{-\frac{1+m}{2}d} \;{[j, m-2+j]}\; g  )\\
           &=\frac{1}{m}x^{(1-m)d}(\sum_{j=0}^{m-1}\gamma^{-j}\;{]j-1, j-1[}\;)\\
           &=1 \quad (\,\textrm{by Lemma}\ \ref{ce1})\\
           &=\epsilon(u_{0}).
\end{align*}
And,
\begin{align*}
(\Id*S)(u_0)&=\sum_{j=0}^{m-1}\gamma^{-j^2}u_j S(x^{-jd}g^ju_{-j}) \\
           &=\sum_{j=0}^{m-1}\gamma^{-j^2}u_j S(u_{-j})S(g^j)x^{jd} \\
           &=\sum_{j=0}^{m-1}\gamma^{-j^2}u_j  (-1)^{-j}\xi^{j}\gamma^{\frac{j(-j+1)}{2}}x^{-jd+\frac{3}{2}(1-m)d}g^{m+j-1}u_{-j}g^{-j}x^{jd}\\
           &=\sum_{j=0}^{m-1}\gamma^{-j^2}u_j (-1)^{-j}\xi^{j}\gamma^{\frac{j-j^2}{2}+j^2}x^{\frac{3}{2}(1-m)d}g^{m-1}u_{-j}\\
           &=x^{\frac{1-m}{2}d}g^{m-1}\sum_{j=0}^{m-1}(-1)^{-j}\xi^{j}\gamma^{-\frac{j^2+j}{2}}u_ju_{-j}\\
           &=\frac{1}{m} \sum_{j=0}^{m-1}\xi^{2j}\gamma^{-j}\;{[j, m-2+j]}\;\\
           &=\frac{1}{m}\cdot \sum_{j=0}^{m-1}\;{]j-1, j-1[}\;\\
           &=1 \quad (\,\textrm{by the proof of Lemma}\ \ref{ce1})\\
           &=\epsilon(u_{0}).
\end{align*}

For $1\leqslant i \leqslant m-1$,
\begin{align*}
(S*\Id)(u_{i})&=\sum_{j=0}^{m-1}\gamma^{j(i-j)}S(u_j)x^{-jd}g^ju_{i-j}\\
               &=\sum_{j=0}^{m-1}\gamma^{j(i-j)} (-1)^j\xi^{-j}\gamma^{-\frac{j(j+1)}{2}}x^{jd+\frac{3}{2}(1-m)d}g^{m-1-j}u_j  x^{-jd}g^ju_{i-j}\\
               &=x^{\frac{3}{2}(1-m)d} g^{m-1}(\sum_{j=0}^{m-1}(-1)^j\xi^{-j} \gamma^{ij-\frac{j(j+1)}{2}}u_ju_{i-j})\\
               &=(-1)^{-i}\xi^{-i}\gamma^{\frac{i(i+3)}{2}}\frac{1}{m}x^{(1-m)d}y^{i}(\sum_{j=0}^{m-1}\gamma^{-j}\;{[j, m-2-i+j]})\\
               &=(-1)^{-i}\xi^{-i}\gamma^{\frac{i(i+3)}{2}}\frac{1}{m}x^{(1-m)d}y^{i}(\sum_{j=0}^{m-1}\gamma^{-j}\;{]j-1-i, j-1[}) \\
               &=0  \quad (\,\textrm{by Lemma}\ \ref{ce5}) \\
               &=\epsilon(u_{i}),
\end{align*}
\begin{align*}
(\Id*S)(u_{i})&=\sum_{j=0}^{m-1}\gamma^{j(i-j)}u_jS(u_{i-j})g^{-j}x^{jd}\\
               &=\sum_{j=0}^{m-1}\gamma^{j(i-j)}u_j (-1)^{i-j}\xi^{j-i}\gamma^{-\frac{(i-j)(i-j+1)}{2}}x^{(i-j)d+\frac{3}{2}(1-m)d}g^{m+j-i-1}u_{i-j}g^{-j}x^{jd}\\
               &=x^{id+\frac{1-m}{2}d}g^{m-1-i}(\sum_{j=0}^{m-1} (-1)^{i-j}\xi^{j-i}\gamma^{-\frac{i(i+1)+j(j+1)}{2}}u_ju_{i-j})\\
               &=\xi^{-2i}\gamma^{i(i+1)}\frac{1}{m} x^{(i-m)d}y^{i}g^{m-i}(\sum_{j=0}^{m-1} \xi^{2j}\gamma^{-(i+1)j}[j,m-2-i+j])\\
               &=\xi^{-2i}\gamma^{i(i+1)}\frac{1}{m} x^{(i-m)d}y^{i}g^{m-i}(\sum_{j=0}^{m-1} \gamma^{-ij}\;{[j, m-2-i+j]})\\
               &=\xi^{-2i}\gamma^{i(i+1)}\frac{1}{m} x^{(i-m)d}y^{i}g^{m-i}(\sum_{j=0}^{m-1} \gamma^{-ij}\;{]j-1-i, j-1[})\\
               &=0  \quad (\,\textrm{by Lemma}\ \ref{ce5})\\
                &=\epsilon(u_{i}).
\end{align*}

By steps 1, 2, 3, 4, $D(m, d, \xi)$ is a Hopf algebra.\qed

\subsection{}{\bf Properties of $D(m, d, \xi)$.}

For short, let $D:=D(m,d,\xi)$. For this Hopf algebra, the following observation is not hard.
\begin{lemma}\label{d4.3} $D(m, d, \xi)$ is PI,  affine and has GK-dimension one.
\end{lemma}
\begin{proof} See Remark \ref{r4.1}.
\end{proof}

\begin{lemma}\label{d4.6} \emph{gl.dim}$D(m, d, \xi)=1$.
\end{lemma}
\begin{proof} Now we find that $D=\bigoplus_{0\leqslant i\leqslant 2m-1} D_i^l$ satisfies all conditions stated in \cite[Proposition 5.1]{BZ} and thus by \cite[Proposition 5.1(a)]{BZ} every nonzero homogeneous element is regular (and thus a nonzero-divisor). In particular,  $y\in D_0^l$ is a regular element of $D$. It is not hard to see that $(y)$ is a Hopf ideal. Let $D':=D/(y)$. Then $D'$ is a finite
dimensional Hopf algebra. We will show that $D'$ is semisimple at first. For notational convenience, the images of $x,g,u_{i}$ in $D'$ are still
written as $x,g$ and $u_i$. One can check that
$$\int_{D'}^l:=\sum_{i=0}^{md-1}\sum_{j=0}^{m-1}x^ig^j+\sum_{i=0}^{md-1}\sum_{j=0}^{m-1}x^ig^ju_0$$
is a non-zero left integral of $D'$. Indeed, it is not hard to see that $x\int_{D'}^l=g\int_{D'}^l=\int_{D'}^l$ and  the following relations $$x^{md}=g^m=1, \ u_ju_0=0,\ u_i=\gamma^{-i}x^du_i,\ xu_i=u_ix^{-1},\ u_ig=\gamma^i x^{-2d}gu_i$$
$$ \ u_0^2=\frac{1}{m}x^{-\frac{1+m}{2}d}\phi_0\cdots \phi_{m-2}g
$$hold in $D'$ for all $0
\leqslant i \leqslant m-1$ and $1\leqslant j \leqslant m-1$. Here we only explain the equation $u_i=\gamma^{-i}x^du_i$ since the others are clear. Since $(1-\gamma^{-i}x^d)u_i=\phi_{i-1}u_i=yu_{i-1}=0$ in $D'$, we have $u_i=\gamma^{-i}x^du_i$. Therefore,

\begin{align*}u_0\cdot \int_{D'}^l&=
\sum_{i=0}^{md-1}\sum_{j=0}^{m-1}x^ig^ju_0+\sum_{i=0}^{md-1}\sum_{j=0}^{m-1}x^ig^ju_0^{2}\\
&=\sum_{i=0}^{md-1}\sum_{j=0}^{m-1}x^ig^ju_0+\sum_{i=0}^{md-1}\sum_{j=0}^{m-1}x^ig^j\\
&=\epsilon(u_0)\cdot \int_{D'}^l,
\end{align*}
and
\begin{align*}u_s\cdot \int_{D'}^l&=
\sum_{i=0}^{md-1}\sum_{j=0}^{m-1}\gamma^{sj}x^ig^ju_s\\
&=\sum_{i=0}^{md-1}\sum_{j=0}^{m-1}\gamma^{sj}x^ig^j\gamma^{-s}x^{d}u_s\\
&=\gamma^{-s}\sum_{i=0}^{md-1}\sum_{j=0}^{m-1}\gamma^{sj}x^ig^ju_s,
\end{align*}
which implies that $\sum\limits_{i=0}^{md-1}\sum\limits_{j=0}^{m-1}\gamma^{sj}x^ig^ju_s=0$, and so
$u_s\cdot \int_{D'}^l=0=\epsilon(u_{s})\int_{D'}^l$
for all $1 \leqslant s \leqslant m-1$.
Clearly,  $\epsilon(\int_{D'}^l)=2m^{2}d\neq
0$. So $D'$ is semisimple.

Secondly, let $M$ be the trivial $D$-module $k$. By \cite{LL} or \cite[Corollary 1.4]{BG},
it is enough to show that p.dim$_DM$=1. Since $y$ is a regular
element (i.e., not a zero divisor) of $D$ and $yM=0$, $M$ cannot be a
submodule of a free $D$-module, which implies p.dim$_DM\neq 0$.
Since  $D'$ is semisimple, we get p.dim$_{D'}M = 0$. By standard method of ``change of rings" (see e.g., \cite[Lemma 5.26]{Lam}), p.dim$_DM$=1.
\end{proof}

\begin{lemma}\label{d4.4} $\io(D)=2m, \im(D)=m$.
\end{lemma}
\begin{proof} By the proof of Lemma \ref{d4.6} $D'$ is semisimple, then $D'$ is unimodular. Note that the left and right homological integrals agree with the classical left and right integrals respectively when the Hopf algebra is of finite dimensional. So $D'$ is also unimodular for homological integrals, that is, the left homological integral of $D'$, also denoted by $\int_{D'}^l$, is isomorphic to $k$ as $D'$-bimodule. Thus we have $\int_{D'}^l=D'/(x-1, g-1, u_0-1, u_1, \cdots, u_{m-1})$. Using \cite[Lemma 2.6]{LWZ},  $$\int_D^l=(\int_{D'}^l)^{\tau^{-1}}=D/(y, x-1,
g-\gamma^{-1}, u_0-\xi^{-1}, u_1, u_2, \cdots, u_{m-1}),$$ where $\tau$ is the algebra automorphism of $D$ such that $yh=\tau(h)y$ for all $h\in D$. The
corresponding homomorphism $\pi$ yields left and right winding
automorphisms
\[{\Xi_{\pi}^l:}
\begin{cases}
x\longmapsto x, &\\
y\longmapsto y, &\\
g\longmapsto \gamma^{-1}g, &\\
u_i\longmapsto \xi^{-1}u_i, &
\end{cases}
\textrm{and}\;\;\;\; \Xi_{\pi}^r:
\begin{cases}
x\longmapsto x, &\\
y\longmapsto \gamma^{-1}y. &\\
g\longmapsto \gamma^{-1}g, &\\
u_i\longmapsto \xi^{-(2i+1)} u_i. &
\end{cases}\]
Clearly these automorphisms have order $2m$ and $G_\pi^l\cap
G_\pi^r=\langle (\Xi_{\pi}^l)^m \rangle$. Then we have $\io(D)=2m$
and $\im(D)=m$ by $|G_\pi^l\cap G_\pi^r|=2$.
\end{proof}

\begin{remark} These winding automorphisms are just the automorphisms constructed in Subsection 4.1 and  the corresponding gradings on $D$  have been exhibited there.
\end{remark}

\begin{lemma}\label{d4.7} $D(m, d, \xi)$ is prime.
\end{lemma}
\begin{proof} We know that $G_\pi^l$ is a finite abelian group acting faithfully on $D(m, d,
\xi)$ (see the proof of Lemma \ref{d4.4}). Moreover, the $\widehat{G_\pi^l}$-grading is strong  and $D_0^l=k[x^{\pm 1}, y]$ is a commutative domain, which shows
that $D$ meets the initial conditions of \cite[Proposition 5.1]{BZ}.
It follows that PI-deg$(D)\leqslant \io(D)=2m$. Since $D$ is regular,
we have $\io(D)\leqslant$ PI-deg$(D/P_0)$ by \cite[Lemma
5.3]{LWZ}, where $P_0$ is the minimal prime ideal of $D$ contained
in $\ker(\epsilon)$. It is clear that PI-deg $(D/P_0)\leqslant$
PI-deg$(D)$. So
PI-deg$(D)=\io(D)=2m$ and thus \cite[Proposition 5.1 (d)]{BZ} applied.
Therefore $D$ is prime.
\end{proof}

The next proposition is a direct consequence of Lemmas \ref{d4.3}, \ref{d4.6}, \ref{d4.4} and  \ref{d4.7}.
\begin{proposition} $D(m, d, \xi)$ is an affine prime regular
Hopf algebra of GK-dimension one with $\io(D)=2m, \im(D)=m$.
\end{proposition}

At the end of this section, we show that

\begin{proposition} $D$ is not pointed.
\end{proposition}
\begin{proof} Let $f:\; D\to D'=D/(y)$ be the canonical Hopf epimorphism. We need to show that $u_i\neq 0$ in $D'$ for all $i$ firstly. Since $y$ is a normal element of $D$, $u_i=0$ in $D'$ forces $u_i=yz_i$ for some $z_i$ in $D$. By the bigrading structure of $D$ and the fact the $y$ is nonzero-divisor, one can see that $z_i=\alpha_iu_{i-1}$ for some $\alpha_i\in D_{00}=k[x^{\pm 1}].$ Therefore,
$$u_i=yz_i=y\alpha_iu_{i-1}=\alpha_iyu_{i-1}=\alpha_i\phi_{i-1}u_i,$$ which contradicts the fact that $\phi_{i-1}=1-\gamma^{-i}x^d$ is not invertible in $[x^{\pm 1}]$.
Secondly, by \cite[Corollary 5.3.5]{Mo}, $D'$ is pointed
if $D$ is pointed.  We will show that $D'$ is not pointed. Otherwise, it is semisimple and pointed and so it is cosemisimple and pointed by \cite{LR}. This implies that $D'$ is cocommutative which is absurd. Therefore, $D$ is not pointed.
\end{proof}

A direct consequence of this proposition is that $D(m,d,\xi)$ is not isomorphic to any one of Hopf algebras listed in Subsection 2.3.
\begin{corollary} As a Hopf algebra, $D(m,d,\xi)$ is not isomorphic to any one of algebras listed in Subsection 2.3.
\end{corollary}
\begin{proof} All of Hopf algebras given in Subsection 2.3 are pointed while $D(m,d,\xi)$ is not.
\end{proof}

\begin{remark}  In practice, the assumption ``pointed" is always added  when we want to
classify Hopf algebras of lower GK-dimensions. As a matter of fact, all known examples are pointed
and it is widely believed that, at least for prime regular Hopf algerbas of GK-dimension one, these
Hopf algebras should be pointed automatically. Our new examples will change this naive understanding
since all  the new examples are not pointed!
\end{remark}

\section{The Hopf Subalgebra $\widetilde{H}$}

From now on, $H$ always denotes a  prime regular Hopf algebra of GK-dimension
one with $\io(H)=n>\im(H)=m>1$ and let $t:=n/m$.
The aim of this section is to construct a Hopf subalgebra
$\widetilde{H}$ of $H$ satisfying the conditions of Lemma \ref{l2.8}.

Recall that the letter $\zeta$ denotes
 the primitive $n$th root of $1$ in the definition of the
bigrading of $H$ given in Subsection 2.2:
$$H=\bigoplus_{0\leqslant i,j\leqslant n-1} H_{ij},$$ where $H_{ij}=H_i^l\cap
H_j^r$.

It is not hard to find that  \cite[Lemma 6.3]{BZ} still holds in our case and so we state it without proof.

\begin{lemma}\label{l3.3} Let $H=\bigoplus_{0\leqslant i,j\leqslant n-1} H_{ij}$. Then

\emph{(a)} $S(H_i^l)=H_{-i}^r$ and $S(H_{ij})=H_{-j,-i}$ \emph{(}where the
suffixes are interpreted \emph{mod} n\emph{)}.

\emph{(b)} If $i\neq j$, then $\epsilon(H_{ij})=0$.

\emph{(c)} $\epsilon(H_{ii})\neq 0$.
\end{lemma}

\begin{lemma}\label{l3.1}  Let $\Xi_\pi^l$ and $\Xi_\pi^r$ be the left and right winding automorphisms defined as in Subsection 2.2 respectively. Then $(\Xi_\pi^l)^m=(\Xi_\pi^r)^m.$
\end{lemma}
\begin{proof} Since $G_\pi^l$ and $G_\pi^r$ are cyclic groups, $G_\pi^l\cap
G_\pi^r$ is a cyclic subgroup of both $G_\pi^l$ and $G_\pi^r$. By
$\im(H)=|G_\pi^l/G_\pi^l\cap G_\pi^r|=m$, the order of $G_\pi^l\cap
G_\pi^r$ is $t$. But the order $t$ cyclic subgroup of $G_\pi^l$ and
$G_\pi^r$ is just $\langle(\Xi_\pi^l)^m\rangle$ and
$\langle(\Xi_\pi^r)^m\rangle.$ So there exists an integer $s,
1\leqslant s \leqslant t-1$, such that
$(\Xi_\pi^l)^m=(\Xi_\pi^r)^{sm}$. Now we show that $s=1$. Since $\epsilon(H_{11})\neq 0$ by Lemma \ref{l3.3}, we have $H_{11}\neq 0$. Then there exists
 $0\neq x \in H_{11}=\{x\in H|\Xi_\pi^l(x)=\zeta x,
\Xi_\pi^r(x)=\zeta x\}$, so
$$\zeta^mx=(\Xi_\pi^l)^m(x)=(\Xi_\pi^r)^{sm}(x)=\zeta^{sm}x.$$
Therefore $\zeta^m=\zeta^{sm}$, and hence $s=1$.
\end{proof}

\begin{lemma}\label{l3.2} For every $j$ with $1\leqslant j\leqslant n-1$, $H_{0j}\neq 0$ if and only if $j\equiv 0$ \emph{(mod} $t$\emph{)} for all $0\leqslant j\leqslant n-1$.
\end{lemma}
\begin{proof} Let $0\neq x\in H_{0j}$. Note that $H_{0j}=\{x\in H|\Xi_\pi^l(x)=x,
\Xi_\pi^r(x)=\zeta^jx\}$. By the above lemma, we have
$$x=(\Xi_\pi^l)^m(x)=(\Xi_\pi^r)^{m}(x)=\zeta^{jm}x.$$ This
implies $\zeta^{jm}=1$. So we get $j\equiv 0$ (mod
$t$). That is, if $j\not\equiv 0$ (mod $t$) then $ H_{0j} = 0$. Therefore
we can write
$$H_0^l=\bigoplus_{0\leqslant j \leqslant m-1}H_{0, jt}.$$ Now
it remains to show that each $H_{0, jt}\neq 0$ for all $0\leqslant j
\leqslant m-1$.

Let $s$ be the minimum positive integer such that $H_{0, st}\neq 0$. Such $s$ must exist. Indeed, if there is no such $s$, that is, $H_{0, jt}=0$ for all $1\leqslant j\leqslant
m-1$. Then $H_0^l=H_0$. Dually, $H_0^r=H_0=H_0^l$. Therefore
$H_0^l$ is a Hopf subalgebra of $H$ by Lemma \ref{l2.5} again. By the proofs of \cite[Propositions 4.2 and  4.3]{BZ}, we see that this can happen only in the following cases:  $H$ is isomorphic as a Hopf algebra either to the Taft algebra $H(n, 0,
\xi)$ or to the infinite dihedral group algebra $k\mathbb{D}$.  In either case,
$\im(H)=1$. This contradicts the hypothesis $\im(H)>1$. Next, we show that $s$ must be a factor of $m$, that is, $s|m$. If not, there exists a positive number $a$ such that $sa<m$ and $s(a+1)>m$. Since $H_{0}^{l}$ is a domain, $0\neq (H_{0,st})^{a+1}\subset H_{0,st(a+1)-mt}$ and thus $s(a+1)-m<s$ is a smaller number such that $H_{0,(s(a+1)-m)t}\neq 0$ which contradicts to the minimality of $s$. Therefore, $H_0^{l}=\bigoplus_{0\leqslant j\leqslant \frac{m}{s}-1}H_{0,jst}$ (by using the minimality of $s$ again) and thus to show the result it is enough to show that $s=1$.

We claim that $$H=\bigoplus_{0 \leqslant i\leqslant n-1,  0\leqslant j\leqslant \tfrac{m}{s}-1}H_{i, i+jst}.\;\;\;\;\;\;(\ast)$$
 At first, since $H_{ii}\neq 0$ for all $i$ by Lemma \ref{l3.3} and every nonzero homogeneous element is nonzero-divisor by \cite[Proposition 5.1(a)]{BZ}, we have $H_{i, i+jst}\supseteq H_{ii}H_{0, jst}\neq 0$ for all $0 \leqslant i\leqslant n-1$ and $0 \leqslant j\leqslant \tfrac{m}{s}-1$. Secondly, we show that $H_{ik}\neq 0$ if and only if $st|(i-k)$. We already know the ``if" part and we only need to show the ``only if" part. In fact, suppose $H_{ik}\neq 0$ and let $0\neq x\in H_{ik}=\{h\in H|\Xi_\pi^l(h)=\zeta^ih,
\Xi_\pi^r(h)=\zeta^kh\}$. Applying Lemma \ref{l3.1}, we have
$$\zeta^{im}x=(\Xi_\pi^l)^m(x)=(\Xi_\pi^r)^{m}(x)=\zeta^{km}x.$$
That is, $\zeta^{im}=\zeta^{km}$ and so  $i-k\equiv 0$ (mod $t$). So it is harmless to assume that $k=i+at$
for $a$ a positive number and we need to show that $s|a$. Applying \cite[Proposition 5.1(a)]{BZ} again, $H_{0,at}\supseteq H_{i, i+at}H_{n-i, n-i}\neq 0$ and thus $s|a$ since we already know that $H_0^{l}=\bigoplus_{0\leqslant j\leqslant \frac{m}{s}-1}H_{0,jst}$. This proves the equation $(\ast)$. By the equation $(\ast)$, it is not hard to see that $\im(H)=m/s$. Therefore, $s=1$.
\end{proof}

The proof of this lemma indeed implies the following result.
\begin{proposition}\label{p3.5} For every $i, j$ with $1\leqslant i, j\leqslant n-1$, $H_{ij}\neq 0$ if and only if $i-j\equiv 0$ \emph{(mod} $t$\emph{)} for all $0\leqslant i, j\leqslant n-1 $ .
\end{proposition}

This proposition also implies that
$$H_{it}^l=\bigoplus_{0\leqslant j\leqslant m-1 }H_{it, jt}\quad \text{and}\quad H_{jt}^r=\bigoplus_{0\leqslant i\leqslant m-1 }H_{it, jt}$$
for all $0\leqslant i, j\leqslant m-1$. Thus we can define
$\widetilde{H}$ as follows:
\begin{equation}\label{eq5.1}\widetilde{H}:=\bigoplus_{0\leqslant i, j\leqslant m-1 }
H_{it, jt}=\bigoplus_{0\leqslant i\leqslant m-1 }
H_{it}^l=\bigoplus_{0\leqslant j\leqslant m-1 } H_{jt}^r.\end{equation}

\begin{lemma} $\widetilde{H}$ is a Hopf subalgebra of $H$.
\end{lemma}
\begin{proof} We need to verify the algebra structure,
coalgebra structure and the antipode condition of $\widetilde{H}$.
Clearly, by the bigraded structure of $H$, it is easy to
see that $\widetilde{H}$ is an algebra. Since each $H_i^l$ is a
right coideal of $H$ and $H_j^r$ is a left coideal of $H$ for all
${0\leqslant i, j\leqslant n-1 }$ (by Lemma \ref{l2.5}),
$\Delta(H_{it,jt})\subseteq\sum_{k,l}H_{kt}^l\otimes
H_{lt}^r\subseteq \widetilde{H}\otimes\widetilde{H}$. So
$\widetilde{H}$ is a coalgebra. Note that $S(H_{it, jt})=H_{-jt,
-it}\subset \widetilde{H}$ by Lemma \ref{l3.3}, thus $\widetilde{H}$ is a
Hopf subalgebra of $H$.
\end{proof}

\begin{lemma} $\widetilde{H}$ is regular and
\emph{gl.dim}$(\widetilde{H})=1$.
\end{lemma}
\begin{proof} Deducing from the fact that $H=\bigoplus_{0\leqslant i\leqslant n-1}H_i^l=\bigoplus_{0\leqslant j\leqslant
n-1}H_j^r$ is a strongly graded algebra,  we get a new grading for $H$: $$H=\bigoplus_{0\leqslant
i\leqslant t-1}\widetilde{H} H_i^l=\bigoplus_{0\leqslant i\leqslant
t-1}\widetilde{H} (H_1^l)^i.$$ Since $H_1^l$ is an affine invertible
$H_0^l$-bimodule, $\widetilde{H}H_i^l$ is an affine invertible
$\widetilde{H}$-bimodule. Therefore, $H$ is a
projective $\widetilde{H}$-module and $\widetilde{H}$ is a direct
summand of $H$ as an $\widetilde{H}$-module. Thus, if $V$ is any
$\widetilde{H}$-module, then
$$\text{p.dim}_{\widetilde{H}}(V)\leqslant
\text{p.dim}_{\widetilde{H}}(H\otimes_{\widetilde{H}}V) \leqslant
\text{p.dim}_H(H\otimes_{\widetilde{H}}V),$$ where the second
inequality holds because $H$ is a projective $\widetilde{H}$-module.
Therefore, gl.dim$(\widetilde{H})\leqslant 1$. Since $\widetilde{H}$ is not semisimple (otherwise, it is finite
dimensional), gl.dim$\widetilde{H}=1$.
\end{proof}

\begin{lemma} $\widetilde{H}$ is prime and
PI-deg\emph{(}$\widetilde{H}$\emph{)}=$m$.
\end{lemma}
\begin{proof}
Note that $\widetilde{H}=\bigoplus_{0\leqslant i\leqslant m-1 }
H_{it}^l$. Therefore, the lemma follows directly from \cite[Corollary 5.1 (b)]{BZ}.
\end{proof}

\begin{proposition} \label{t3.9} Let $\widetilde{H}$ be the algebra constructed as \eqref{eq5.1}. Then it is isomorphic as a Hopf algebra either to the Taft algebra $H(m, 1,
\xi)$ or to the generalized Liu algebra $B(m, \omega, \gamma)$. As a
consequence, $\io(\widetilde{H})=\im(\widetilde{H})=m$.
\end{proposition}
\begin{proof}By the above three lemmas,
$\widetilde{H}$ is an affine prime regular Hopf algebra, and it is
clear that $\widetilde{H}$ is of GK-dimension one. Denote the restriction
of the actions of $\Xi_\pi^l$ and $\Xi_\pi^r$ to $\widetilde{H}$ by
$\Gamma^l$ and $\Gamma^r$, respectively. Since
$\widetilde{H}=\bigoplus_{0\leqslant i\leqslant m-1 } H_{it}^l$, we
can see that for each $0\leqslant i\leqslant m-1$ and any $0\neq x
\in H_{it}^l$,
$$(\Gamma^l)^m(x)=\zeta^{itm}x=x.$$
This implies that the group $\langle\Gamma^l\rangle$ has order $m$.
Similarly, $|\langle\Gamma^r\rangle|=m$. We claim that
$$\langle\Gamma^l\rangle\cap \langle\Gamma^r\rangle={1}.$$
In fact, if $(\Gamma^l)^i=(\Gamma^r)^j$ for some $0\leqslant i,
j\leqslant m-1$. Choose $0\neq x\in H_{tt}$, we find
$$\zeta^{ti}x=(\Gamma^l)^i(x)=(\Gamma^r)^j(x)=\zeta^{tj}x$$ which implies
$i=j$. Let $0\neq y\in H_{0,t}$, then
$$y=(\Gamma^l)^i(y)=(\Gamma^r)^j(y)=\zeta^{tj}y$$ forces
$j=0$. Thus we get $i=j=0$, i.e., $\langle\Gamma^l\rangle\cap
\langle\Gamma^r\rangle={1}$. Therefore, Lemma \ref{l2.8} is applied. So,
$\io(\widetilde{H})=\im(\widetilde{H})=m$ and
 $\widetilde{H}$ is isomorphic as a Hopf algebra either to
the Taft algebra $H(m, 1, \xi)$ or to the generalized Liu algebra
$B(m, \omega, \gamma)$.
\end{proof}

\section{From $\widetilde{H}$ to $H$}
This section is to build some general relations between  $\widetilde{H}$ and $H$. Our final aim is to show that
$\widetilde{H}$ can determine the structures of $H$ entirely.
We start with the following definition.
\begin{definition} We call $H$ is \emph{primitive} \emph{(}respectively, \emph{group-like}\emph{)} if
$\widetilde{H}$ is primitive, i.e., $\widetilde{H}\cong H(m, 1, \xi)$ \emph{(}respectively, if $\widetilde{H}$ is group-like, i.e., $\widetilde{H}\cong B(m, \omega,\gamma)$\emph{)}.
\end{definition}
By Proposition \ref{t3.9}, $H$ is either primitive or group-like. By definition, $\widetilde{H}=\bigoplus_{0\leqslant i,
j\leqslant m-1 } H_{it, jt}$ is bigraded automatically. Meanwhile, $\widetilde{H}$ has left integrals $\int_{\widetilde{H}}^l$ and
from which we also can construct another bigrading on $\widetilde{H}$ just as in Subsection 2.2.

\begin{lemma}\label{l6.2} Under a suitable choice of $\zeta$, these two bigradings on $\widetilde{H}$  are the same.
\end{lemma}
\begin{proof} The proof is indeed implicit in the arguments given in \cite[Section 6]{BZ}.
For completeness, we prove it here.
We start with a general form.
Let $A$ be a prime regular Hopf algebra
of GK-dimension one with an algebra homomorphism $\mu: A\to k$ (note that
$\mu$ need not be the algebra homomorphism $\pi:A\to
A/\text{r.ann}(\int_A^l$)). Denote the groups of the left and
right winding automorphisms $\Xi_\mu^l$ and $\Xi_\mu^r$ associated
to $\mu:A\to k$ by $G_\mu^l$ and $G_\mu^r$ respectively. Fix a
primitive $n$th root $\zeta$ of $1$, and define $\chi\in
\widehat{G_\mu^l}$ and $\eta\in \widehat{G_\mu^r}$ by setting
$$\chi(\Xi_\mu^l)=\zeta \quad \text{and} \quad
\eta(\Xi_\mu^r)=\zeta.$$ Assume that $G_\mu^l$ and $G_\mu^r$ satisfy
$$|G_\mu^l|=\text{PI-deg}(A)=n \quad \text{and} \quad G_\mu^l \cap
G_\mu^r=\{1\} \quad \quad  \quad \quad  \quad (\star).$$  Then
there is a bigrading on
$$A=\bigoplus_{0\leqslant i, j\leqslant n-1 } A_{ij},$$
where $A_{ij}=\{a\in A|\Xi_\mu^l(a)=\zeta^ia,\;\Xi_\mu^r(a)=\zeta^ja\}$. It is proved in \cite[Section 6]{BZ} that
$A\cong H(n,1,\xi)$ or $A\cong B(n,\omega,\gamma)$ and under a suitable choice of $\zeta$, $A_{ij}$ is exactly
$H(n,1,\xi)_{ij}$ or $B(n,\omega,\gamma)_{ij},$ where $H(n,1,\xi)_{ij}$ and $B(n,\omega,\gamma)_{ij}$ are homogeneous
components of the bigrading constructed in Subsection 2.2. Now the proof is completed by noting that the algebra homomorphism $\mu=\pi|_{\widetilde{H}}:\widetilde{H}\to k$ (where $\pi:H\to H/\text{r.ann}(\int_H^l)$) induces the left and right winding automorphisms satisfying the condition $(\star)$ by Proposition \ref{t3.9}.
\end{proof}

\begin{remark}\label{r4.3}By Lemma \ref{l6.2},  we can freely use the calculation results of
$H_{ij}$ to $$\widetilde{H}:=\bigoplus_{0\leqslant i, j\leqslant m-1
} H_{it, jt}$$ as in Subsection 2.3. That is, if $H$ is primitive,
then we can assume that $H_{it, jt}=k[x^m]x^{j-i}g^i$; and if $H$
is group-like, then we can take that $H_{it, jt}=k[x^{\pm
1}]y^{j-i}g^i$.
\end{remark}


Now we denote the fraction field of $H_0^l$ by $Q_0^l$.
Brown and Zhang \cite[Section 5]{BZ} showed that there is a delicate $\widehat{G_\pi^l}$-action on $Q_0^l$ defined as follows: For each $\chi^i\in \widehat{G_\pi^l}$ and $a\in Q_0^l$,  $$\kappa_{i}(a):=u_{i}au_{i}^{-1}.$$ Here, to avoid confusion, we denote  the automorphism of $Q_0^l$ corresponding
to $\chi^i\in \widehat{G_\pi^l}$ by $\kappa_i$. And since $H_i^{l}$
is an invertible $H_0^{l}$-module, $Q_0^lH_i^{l}=Q_0^lu_{i}$ for any non-zero element $u_i\in H_i^l$. Brown and Zhang proved that this action is independent of the choice of the non-zero element $u_{i}$ and clearly satisfies
\begin{equation} xa=\kappa_i(a)x
\end{equation} for $x\in H_i^l$ and $a\in H_0^l$.  Furthermore, \cite[Proposition 5.1 (b)]{BZ} implies that the restriction of
$\kappa_{i}$ to $H_0^{l}$ is still an automorphism and we denote this restriction by $\kappa_i^l$ for convenience. In a special situation, the following lemma is implicit in \cite[Subsection 6.2]{BZ}.

\begin{lemma}\label{l4.3} $K^l:=\{\kappa_i^l\}_{0\leqslant i\leqslant n-1}$ is a cyclic group.
\end{lemma}
\begin{proof} We only need to show $\kappa_i^l=(\kappa_1^l)^i$ for all
$i$. Let $x_1, x'_1$ be any non-zero elements of $H_1^l$, for any
$a\in H_0^l$,
$$x_1x'_1a=x_1\kappa_1^l(a)x'_1=\kappa_1^l(\kappa_1^l(a))x_1x'_1.$$
Applying the strongly graded property of $H$, it follows that
$\kappa_2^l=(\kappa_1^l)^2$. Similarly, $\kappa_i^l=(\kappa_1^l)^i$
for all $2\leqslant i\leqslant n-1$. So $K^l$ is a cyclic group.
\end{proof}

Meanwhile, $G_\pi^r$ acts on $H_0^l$ as explained in Subsection 2.2. We denote by
$\rho^l$ the resulting map from $G_\pi^r$ to Aut($H_0^l$). Let
$P_\pi^l:=\langle\rho^l(G_\pi^r), K^l\rangle\subseteq
\text{Aut}(H_0^l)$. By \cite[Proposition 5.2]{BZ}, $P_\pi^l$ is
abelian and $P_\pi^l=K^l$. So the cyclic group $\rho^l(G_\pi^r)$ is
a subgroup of $K^l$. By \cite[Proposition 5.1]{BZ}, $Z(H)\subseteq H_0^l$. Similarly, $Z(H)\subseteq H_0^r$.
Therefore, $Z(H)\subseteq H_0$ and thus $H_0$ is a $Z(H)$-module naturally. One of our key observations is the
following lemma.

\begin{lemma}\label{l4.5} $H_0$ is a torsion-free $Z(H)$-module of rank at least $t$.
\end{lemma}
\begin{proof}
Since PI-deg$(H)=\io(H)=n$, $K:=\{\kappa_i\}_{0\leqslant i\leqslant
n-1}$ acts faithfully on $Q_0^l$ by \cite[Proposition 5.1]{BZ}. So
it is easy to see that $K^l$ acts faithfully on $H_0^l$. Then we
have a $\widehat{K^l}$-grading on $H_0^l$:
$$H_0^l=\bigoplus_{\chi^i\in\widehat{K^l}}(H_0^l)_i^{\kappa},$$
where $(H_0^l)_i^\kappa=\{x\in H_0^l\ |\ g(x)=\chi^i(g)x, \forall
g\in K^l\}.$

Note that the grading of $H_0^l$ induced by the action
of $\rho^l(G_\pi^r)$ is just $H_0^l=\bigoplus_{0\leqslant i\leqslant
m-1}H_{0,it}$. Since $\rho^l(G_\pi^r)$ is a subgroup of the cyclic group
$K^l$, $H_0^l=\bigoplus_{\chi^i\in
\widehat{K^l}}(H_0^l)_i^\kappa$ is a refinement of
$H_0^l=\bigoplus_{0\leqslant i\leqslant m-1}H_{0,it}$. Therefore
$$H_0=(H_0^l)_0^\kappa\bigoplus (H_0^l)_m^\kappa \bigoplus \cdots
\bigoplus (H_0^l)_{(t-1)m}^\kappa.$$ Because of the faithful
action of $K^l$ on $H_0^l$, each $(H_0^l)_i^\kappa$ is torsion-free
on $(H_0^l)_0^\kappa$ of rank at least one. By \cite[Proposition 5.1]{BZ},
$(H_0^l)_0^\kappa=Z(H)$.  Thus $H_0$ is
torsion-free over $Z(H)$ of rank at least $t$.
\end{proof}

With this observation, we have the following conclusion.

\begin{proposition}\label{p4.6} Each homogeneous component $H_{i, i+jt}$ of $H=\underset{0\leqslant i\leqslant n-1, 0\leqslant j\leqslant m-1}{\bigoplus} H_{i, i+jt}$ is a free $H_0$-module of
rank one on both sides.
\end{proposition}
\begin{proof} We work with left modules. By
\cite[Proposition 5.1]{BZ},  each $H_i^l$ is a torsion-free
$H_0^l$-module.  So the rank of $H_i^l$ on $H_0^l$ is at least one.
By Proposition \ref{p3.5}, each homogeneous $H_{i, i+jt}$ is a
non-zero $H_0$-module, then it is torsion-free of rank at least one.
Since $\widetilde{H}$ is isomorphic as a Hopf algebra either to the
Taft algebra $H(m, 1, \xi)$ or to the generalized Liu algebra $B(m,
\omega, \gamma)$,  $H_0$ is isomorphic either to $k[x^m]$ or to
$k[x^{\pm 1}]$. Both of them are principal ideal domains, so each
$H_{i, i+jt}$ is free on $H_0$ of rank at least one which implies
the rank of $H$ over $H_0$ is not smaller than $nm$. By Lemma
\ref{l4.5}, the rank of $H_0$ over $Z(H)$ is at least $t$.
Therefore, $$\text{rank}_{Z(H)}H \geqslant nmt=n^2.$$ Recall that
the PI-degree of $H$ is $n$ and equals the square root of the rank of
$H$ over $Z(H)$.  So the rank of $H_0$ over $Z(H)$ is $t$ and each
$H_{i, i+jt}$ is a free $H_0$-module of rank one.
\end{proof}

Since $H_0=k[x^m]$ or $H_0=k[x^{\pm 1}]$, there is a generating set
$\{u_{i, i+jt}|0\leqslant i\leqslant n-1, 0\leqslant j\leqslant m-1
\}$ of these $H_{i, i+jt}$ satisfying
$$u_{00}=1\;\;\;\;\textrm{and}\;\;\;\;H_{i, i+jt}=u_{i, i+jt}H_0=H_0u_{i, i+jt}.$$
So, $H$ can be written as
\begin{equation}\label{eq6.2} H=\underset{0\leqslant j\leqslant m-1}{\bigoplus_{0\leqslant i\leqslant n-1}}H_0u_{i,
i+jt}=\underset{0\leqslant j\leqslant m-1}{\bigoplus_{0\leqslant
i\leqslant n-1}}u_{i, i+jt}H_0.\end{equation}

With these preparations,  we can analyse the structure of $H$ further by dichotomy now: either $\widetilde{H}$
is $H(m, 1, \xi)$, the primitive case; or $\widetilde{H}$ is
$B(m, \omega, \gamma)$, the group-like case.

\section{Primitive case}

If $H$ is primitive,  $\widetilde{H}=H(m, 1, \xi)$. We will prove $H$ is isomorphic
as a Hopf algebra to $H(n, t, \theta)$, for $\theta$ some primitive
$n$th root of 1. Recall that
$$\widetilde{H}=H(m, 1, \xi)=k\langle g, x|g^m=1,xg=\xi gx\rangle.$$ Note that
$H=\bigoplus_{0\leqslant i\leqslant n-1, 0\leqslant j\leqslant m-1}
H_{i, i+jt}$, $\widetilde{H}=\bigoplus_{0\leqslant i, j\leqslant
m-1} H_{it, jt}$ and $H_{it, jt}=k[x^m]x^{j-i}g^{i}$ (the index $j-i$
is interpreted mod $m$). In particular, $H_0=k[x^m]$, $H_{0,
jt}=k[x^m]x^j$ and $H_{tt}=k[x^m]g$.

By Lemma \ref{l3.3}, $\epsilon(u_{11})\neq 0$. Multiplied with a suitable scalar, we can assume that
$\epsilon(u_{11})=1$ throughout  this section.

\begin{lemma}\label{l4.7} Let $u:=u_{11}$. Then $H_1^l=H_0^lu$, $H=\bigoplus_{0\leqslant i\leqslant
t-1}\widetilde{H}u^i$ and $u$ is invertible.
\end{lemma}
\begin{proof} By the bigraded structure of $H$, we have
$$H_{0t}H_{11}\subseteq H_{1, 1+t},\;\;\;\; H_{0, (m-1)t}H_{1, 1+t}\subseteq
H_{11},$$ which imply $$H_{0t}H_{0, (m-1)t}H_{1, 1+t}\subseteq
H_{0t}H_{11}\subseteq H_{1, 1+t}.$$
Since $H_{0t}H_{0,
(m-1)t}=x^mH_0$ is a maximal ideal of $H_0=k[x^m]$ and $H_{1,1+t}$ is a free $H_0$-module of rank one
 (by Proposition \ref{p4.6}
),  $H_{0t}H_{0,
(m-1)t}H_{1, 1+t}$ is a maximal $H_0$-submodule of $H_{1, 1+t}$.
Thus $H_{0t}H_{11}=H_{0t}H_{0,
(m-1)t}H_{1, 1+t}=x^mH_{1, 1+t}$ or $H_{0t}H_{11}=H_{1, 1+t}$.

If
$H_{0t}H_{11}=x^mH_{1, 1+t}$, then $xu_{11}=x^m\alpha(x^m)u_{1,
1+t}$ for some polynomial $\alpha(x^m)\in k[x^m]$. So
$$x(u_{11}-x^{m-1}\alpha(x^m)u_{1, 1+t})=0.$$
Therefore, $x^m(u_{11}-x^{m-1}\alpha(x^m)u_{1, 1+t})=0$.
Note that each homogenous $H_{i, i+jt}$ is a torsion-free
$H_0$-module, so $$u_{11}=x^{m-1}\alpha(x^m)u_{1, 1+t}.$$
By  assumption, $\epsilon(u_{11})=1$.
But, by definition,   $\epsilon(x)= 0$. This is impossible. So $H_{0t}H_{11}=H_{1, 1+t}$ which implies
that $H_{0,t}u_{11}=H_{1, 1+t}$.

Similarly, we can show that $H_{0,it}u_{11}=H_{1, 1+it}$
for $0\leqslant i\leqslant m-1$.
 Thus $H_1^l=H_0^lu_{11}$. Since $H=\bigoplus_{0\leqslant i\leqslant n-1} H_i^l$ is strongly graded,
$u_{11}$ is invertible and $H_i^l=H_0^lu_{11}^i$ for all $0\leqslant
i\leqslant n-1$. Let $u:=u_{11}$, then we have
$$H=\bigoplus_{0\leqslant i\leqslant t-1}\widetilde{H}u^i.$$
\end{proof}

We are in a position to determine the structure of $H$ now.

 \begin{proposition}\label{p4.8} With above notations, we have
 $$u^t=g, \;\;\;\;xu=\theta ux,$$
where $\theta$ is a primitive $n$th root of $1$.
 \end{proposition}
\begin{proof}
By $H_{0t}u=uH_{0t}$, there exists a polynomial $\beta(x^m)\in k[x^m]$ such that
$$xu=ux\beta(x^m).$$  Then
$$xu^t=u^tx{\beta}'(x^m)$$ for some polynomial ${\beta}'(x^m)\in
k[x^m]$ induced by $\beta(x^m)$. Since $u^t$ is invertible and
$u^t\in H_{t,t}=k[x^m]g$,  $u^t=ag$ for $0\neq a\in k$. By
assumption, $\epsilon(u)=1$ and thus $a=1$. Therefore, $u^t=g$. Since
$xg=\xi gx$,  we have ${\beta}'(x^m)=\xi$. Then it is easy to see that
$\beta(x^m)=\theta\in k$ with $\theta^t=\xi$. Of course,
$\theta^n=1$.

The last job is to show that $\theta$ is a primitive $n$th root of
$1$. Indeed, assume $\theta$ is a primitive $n'$th root of 1. Then $Z(H)=k[x^{n'}]$, the center of $H$. Recall that the PI-degree of $H$ equals the square root of the rank of $H$ over $Z(H)$.
So the equalities $$n^2=nm\;\text{rank}_{Z(H)}H_0=nm n'/m$$ hold. That is, $n'=n$ and $\theta$ is a primitive $n$th root of
$1$.
\end{proof}

\begin{proposition}\label{p4.9} The element $u$ is a group-like element of $H$.
\end{proposition}
\begin{proof} First of all $H_0^r\cong k[x]\cong H_0^l$. Then $H_0^r\otimes H_0^l\cong k[x, y]$ and the only invertible elements in $H_0^r\otimes H_0^l$ are nonzero scalars in $k$. Since $\Delta(u)$ and $u\otimes u$ are invertible, $\Delta(u)(u\otimes u)^{-1}$ is invertible (and hence a scalar). Thus $u$ must be grouplike by noting that $\epsilon(u)=1$.
\end{proof}

The next theorem follows from Lemma \ref{l4.7} and Propositions
\ref{p4.8}, \ref{p4.9} directly.

\begin{theorem}\label{t4.10}  Let $H$ be an affine prime regular Hopf algebra of GK-dimension one satisfying $\io(H)=n>\im(H)=m>1$. If $H$ is primitive, then $H$ is isomorphic
as a Hopf algebra to an infinite dimensional Taft algebra.
\end{theorem}

\section{Group-like case}
If $H$ is group-like, then $\widetilde{H}=B(m, \omega,
\gamma)$ for $\gamma$ a primitive $m$th root of $1$ and $\omega$ a
positive integer. As usual, the generators of $B(m, \omega, \gamma)$ are
denoted by $x^{\pm 1},y$ and $g$. By Remark \ref{r4.3}, we can assume that
 $\widetilde{H}=\bigoplus_{0\leqslant i, j\leqslant m-1}H_{it,
jt}$ with $$H_{it, jt}=k[x^{\pm 1}]y^{j-i}g^i$$ (the index $j-i$ is
interpreted mod $m$). In particular, $H_{0}=k[x^{\pm 1}]$, $H_{0,
jt}=k[x^{\pm 1}]y^j$ and $H_{t, t}=k[x^{\pm 1}]g$.

Set $u_i:=u_{1,1+it} (0\leqslant i\leqslant m-1)$ for convenience.
By the structure of the bigrading of $H$, we have
\begin{equation}\label{4.4} yu_0=\phi_0u_1,\; yu_1=\phi_1u_2, \cdots,
yu_{m-2}=\phi_{m-2}u_{m-1},\;yu_{m-1}=\phi_{m-1}u_0\end{equation}
and
\begin{equation}\label{4.5}u_0y=\varphi_0u_1,\; u_1y=\varphi_1u_2, \cdots,
u_{m-2}y=\varphi_{m-2}u_{m-1},\;u_{m-1}y=\varphi_{m-1}u_0\end{equation}
for some polynomials $\phi_i, \varphi_i\in k[x^{\pm 1}], 0\leqslant i\leqslant
m-1$. With these notions and the equality $y^m=1-x^{\omega}$, we find that
\begin{equation}\label{4.6}(1-x^{\omega})u_0=y^mu_0=\phi_0\phi_1\cdots
\phi_{m-1}u_0\end{equation} and
\begin{equation}\label{4.7} u_0(1-x^{\omega})=u_0y^m=\varphi_0\varphi_1\cdots
\varphi_{m-1}u_0.\end{equation}

\begin{proposition}\label{p4.11} There is no group-like affine prime regular Hopf algebra of GK-dimension
one $H$ satisfying $\io(H)=n>\im(H)=m>1$ and $n/m>2$.
\end{proposition}
\begin{proof} Since $u_iH_0=H_0u_i$, we have $$u_ix=\alpha_i(x^{\pm
1})u_i \ \ \ \ \text{and}\ \ \ \ u_ix^{-1}=\beta_i(x^{\pm 1})u_i$$
for some $\alpha_i(x^{\pm 1}), \beta_i(x^{\pm 1})\in k[x^{\pm 1}]$.
From
$$u_i=u_ixx^{-1}=\alpha_i(x^{\pm
1})u_ix^{-1}=\alpha_i(x^{\pm 1})\beta_i(x^{\pm 1})u_i,$$ we get
$\alpha_i(x^{\pm 1})\beta_i(x^{\pm 1})=1$ and thus $\alpha_i(x^{\pm
1})=\lambda_ix^{a_i}$ for some $0\neq \lambda_i\in k, 0\neq a_i\in
\mathbb{Z}$.

Note that $u_i^t\in H_{t,(1+it)t}=k[x^{\pm 1}]y^{\bar{it}}g$, where
$\bar{it}\equiv it\; (\text{mod}\; m)$. So we have
$u_i^t=\gamma_i(x^{\pm 1})y^{\bar{it}}g$ for some $\gamma_i(x^{\pm
1})\in k[x^{\pm 1}]$. Hence $u_i^t$ commutes with $x$. Applying
$u_ix=\lambda_ix^{a_i}u_i$ to $u_i^tx=xu_i^t$, we get
$\lambda_i^{\sum_{s=0}^{t-1}a_i^s}=1$ and $x^{a_i^t}=x$. If $t$ is odd, $a_i=1$ and if $t$ is even, then $a_i$ is either $1$ or $-1$.

Now we consider the special case $i=0$.  By
$\epsilon(xu_0)=\epsilon(u_0x)\neq 0$, we find that $\lambda_0=1$.

If $a_0=1$, that is $u_0x=xu_0$, then we will see $u_ix=xu_i$ for
all $0\leqslant i\leqslant m-1$. In fact, by
$$\phi_0xu_1=x\phi_0u_1=xyu_0=yxu_0=yu_0x=\phi_0u_1x,$$ we have $u_1x=xu_1$ since
$H_{1,1+t}$ is a torsion-free $H_0$-module. Similarly, $u_ix=xu_i$
for all $0\leqslant i\leqslant m-1$. Then by the strongly graded
structure $u_{i,i+jt}\in H_i^l=(H_1^l)^i$ and $x$ is commutative with $H_1^l$, it is not hard to see that $u_{i,i+jt}x=xu_{i,i+jt}$ for
all $0\leqslant i\leqslant n-1, 0\leqslant j\leqslant m-1$.
Therefore the center $Z(H)\supseteq H_0=k[x^{\pm 1}]$. By \cite[Lemma 5.2]{BZ}, $Z(H)\subseteq H_0$ and thus $Z(H)=H_0=k[x^{\pm 1}]$. This implies that
$$\text{rank}_{Z(H)}H=\text{rank}_{H_{0}}H= nm < n^2,$$ which contradicts  the fact:
the PI-degree of $H$ is $n$ and equals the square root of the rank
of $H$ over $Z(H)$.

If $a_0=-1$, that is $u_0x=x^{-1}u_0$, we can deduce that
$u_{i,i+jt}x=x^{-1}u_{i,i+jt}$ for all $0\leqslant i\leqslant n-1,\ 0\leqslant j\leqslant m-1$ by using  the
parallel proof of the case $a_0=1$. For $s\in \mathbbm{N}$, let $z_s:=x^s+x^{-s}$.
Define $k[z_s|s\geq 0]$ to be the subalgebra of $k[x^{\pm 1}]$ generated by all $z_{s}$. Note that $k[x^{\pm 1}]$ has rank 2 over $k[z_s| s\geqslant 1]$.
Thus $Z(H)\supseteq k[z_s|s\geq 0]$. Using  \cite[Lemma 5.2]{BZ} again, we have $Z(H)= k[z_s|s\geq 0]$. Hence
$$\text{rank}_{Z(H)}H = 2nm\neq n^2$$
since $n/m >2$ by assumption. This contradicts the fact that the PI-deg$H=n$ again.

Combining these two cases, we get the desired result.
\end{proof}

We turn now to consider the case: $\io(H)=2\im(H)=2m$. In this case, $t=2$. As
discussed in the proof of Proposition \ref{p4.11}, if such $H$ exists then the following relations
\begin{equation}\label{rl} u_ix=x^{-1}u_i\ \ (0\leqslant i\leqslant m-1)\end{equation}
 hold in $H$. Using these relations and \eqref{4.6}, we have
 \begin{equation}\label{r2} \varphi_0\cdots \varphi_{m-1}=1-x^{-\omega}.
 \end{equation}To determine the structure of $H$, we need to give some harmless assumptions on the choice of $u_i$ ($0\leqslant i\leqslant m-1$): (1) We assume $\epsilon(u_0)=1$;
 (2) Let $\xi_s:=e^{\frac{2s\pi i}{\omega}}$ and thus $1-x^{\omega}=\Pi_{s\in S}(1-\xi_sx)$, where $S:=\{0,1,
\cdots, \omega-1\}$.
Since \begin{equation}\label{eq8.5}\phi_0\cdots \phi_{m-1}=y^m=1-x^\omega\end{equation} by \eqref{4.6}, we
have
\begin{equation*} \phi_i=k_ix^{c_{i}}\Pi_{s\in S_i}(1-\xi_sx),\end{equation*}
where $k_i\in k$, $c_{i}$ is an integer and $S_i$ is a subset of $S$.
The second assumption is: Take the $u_i$'s such that
 $\phi_i=\prod_{s\in
S_i}(1-\xi_sx)$. Due to Equation \eqref{eq8.5}, this assumption can be realized; (3)
By the strongly graded structure of $H$, the equality $H_2^l=H^{l}_0g$ and
the fact that $g$ is invertible in $H$, we can take $u_{j, j+2i}$ such that
\begin{equation*}
u_{j, j+2i}=
\begin{cases}
 g^{\frac{j-1}{2}}u_i  & \text{if}\ \ j \ \ \text{is odd},\\
 y^ig^{\frac{j}{2}}    & \text{if}\ \ j \ \ \text{is even},
\end{cases}
\end{equation*}
 for all
$2\leqslant j \leqslant 2m-1$. In the rest of this section, we always make these assumptions.

 We still need two notations, which appeared in the proof of Proposition \ref{p4.2}. For
 a polynomial
$f=\sum a_ix^{b_i} \in k[x^{\pm 1}]$, we denote by $\bar{f}$ the polynomial $\sum
a_ix^{-b_i}$. Then by \eqref{rl}, we have $fu_i=u_i\bar{f}$ and $u_if=\bar{f}u_i$ for all $0\leqslant i\leqslant m-1$.
For any $h\in H\otimes H$, we use
 $$h_{(s_1,t_1)\otimes (s_2,t_2)}$$
 to denote the homogeneous part of $h$ in $H_{s_{1},t_{1}}\otimes H_{s_{2},t_{2}} $.
Both these notations will be used frequently in the proof of the next proposition.

\begin{proposition}\label{p8.2} Let $H$ be a prime regular Hopf algebra with $\widetilde{H}=B(m,\omega,\gamma)$. Assume that $\io(H)=2\im(H)>2$, then as a Hopf algebra,
$$H\cong D(m,d,\xi)$$
constructed as in Subsection 4.1, where $m$ divides $\omega$ and $2$ divides $d(m+1)$.
\end{proposition}
\noindent\emph{Proof.} We divide the proof into several steps. \\[2mm]
\emph{Claim 1. We have $m|\omega$ and $yu_i=\phi_i u_{i+1}=\xi x^du_iy$ for $d=\frac{\omega}{m}$ and some $\xi\in k$ satisfying $\xi^{m}=-1$.}

\emph{Proof of Claim 1:} By associativity of the multiplication, we have many equalities:
\begin{align*}
yu_0y^{m-1}&=\phi_0\varphi_1\varphi_2\cdots \varphi_{m-1}u_0\\
           &=\varphi_0\phi_1\varphi_2\cdots \varphi_{m-1}u_0\\
           &\cdots\\
           &=\varphi_0\varphi_1\varphi_2\cdots \phi_{m-1}u_0,
\end{align*}
which imply that $\phi_i\varphi_j=\varphi_i\phi_j$ for all $0\leqslant i,
j\leqslant m-1$. Using associativity again, we have
\begin{align*}
y^mu_0y^{m(m-1)}&=(1-x^\omega)u_0(1-x^\omega)^{m-1}=-x^\omega(1-x^{-\omega})^mu_0\\
           &=-x^\omega(\varphi_0\varphi_1\varphi_2\cdots \varphi_{m-1})^mu_0\\
           &=(\phi_0\varphi_1\varphi_2\cdots \varphi_{m-1})^mu_0\\
           &=(\varphi_0\phi_1\varphi_2\cdots \varphi_{m-1})^mu_0\\
           &\cdots\\
           &=(\varphi_0\varphi_1\varphi_2\cdots \phi_{m-1})^mu_0,
\end{align*}
where the fourth ``$=$", for example, is gotten in the following way: We multiply $u_0$ by one $y$ from left side at first, then multiply it with $y^{m-1}$ from right side, then continue the procedures above. From these equalities, we have
 $$\phi_i^m=-x^\omega\varphi_i^m$$ for all $0\leqslant
i\leqslant m-1$. So $\phi_i=\xi_ix^d\varphi_i$ where $d=\frac{\omega}{m}$ and $\xi_i\in k$ satisfying $\xi_i^m=-1$. Applying
$\phi_i\varphi_j=\varphi_i\phi_j$, we can see $\xi_i=\xi_j$ for all
$0\leqslant i, j\leqslant m-1$, and we write $\xi:=\xi_i$. Therefore
we have $yu_i=\xi x^du_iy$, where $d=\omega/m$ is an integer.\qed

 In the following of the proof, $d$ is fixed to be the number $\omega/m$.

\noindent\emph{Claim 2. We have $u_ig=\lambda_i x^{-2d}gu_i$ for $\lambda_i=\pm \gamma^i$ and $0\leqslant i\leqslant m-1$.}

\emph{Proof of Claim 2:}
Since $g$ is invertible in $H$, $u_ig=\psi_igu_i$ for some
invertible $\psi_i\in k[x^{\pm 1}]$. Then $u_ig^m=\psi_i^mg^mu_i$
yields $\psi_i^m=x^{-2\omega}$. So $\psi_i=\lambda_ix^{-2d}$ for $\lambda_i\in k$ with
$\lambda_i^m=1$. Our last task is to show that $\lambda_i=\pm \gamma^i$. To show this, we need a preparation, that is, we need to show that
$u_iu_j\neq 0$ for all $i,j$. Otherwise, assume that there exist $i_0,j_0\in \{0,\ldots,m-1\}$ such that
$u_{i_0}u_{j_0}=0$. Using Claim 1, we can find that $u_{i_0}u_j\equiv 0$ and $u_iu_{j_0}\equiv 0$ for all $i,j$. Let $(u_{i_0})$ and $(u_{j_0})$ be the ideals generated by $u_{i_0}$ and $u_{j_0}$ respectively. Then it is not hard to find that
$(u_{i_0})(u_{j_0})=0$ which contradicts  $H$ being prime. So we always have
\begin{equation}\label{r3} u_iu_j\neq 0
\end{equation} for all $0\leqslant i,j\leqslant m-1$.

 Applying this observation, we have $0\neq
u_i^2\in H_{2,2+4i}=k[x^{\pm 1}]y^{2i}g$,
$u_i^2g=\psi_i\overline{\psi_i}gu_i^2=\gamma^{2i}gu_i^2$. Thus
$\psi_i=\pm \gamma^ix^{-2d}$ which implies that $\lambda_i=\pm \gamma^i$. The proof is ended.\qed

We can say more about $\lambda_i$ at this stage.  By $0\neq
u_iu_jg=\gamma^{i+j}gu_iu_j$, we know that $\psi_i=\gamma^ix^{-2d}$ for all $i$ or
$\psi_i=-\gamma^ix^{-2d}$ for all $i$. So
 \begin{equation}\label{r5} \lambda_{i}=\gamma^{i}\ \ \textrm{or}\ \ \lambda_{i}=-\gamma^{i}
 \end{equation} for all $0\leqslant i\leqslant m-1$. In fact, we will show that  $\psi_i=\gamma^ix^{-2d}$ for all $i$ later.

\noindent\emph{Claim 3. For each $0\leqslant i\leqslant m-1$, there are $f_{ij},h_{ij}\in k[x^{\pm 1}]$ with $h_{ij}$ monic such that \begin{equation}\label{r4}\D(u_i)=\sum_{j=0}^{m-1}f_{ij}u_j\otimes
h_{ij}g^ju_{i-j},\end{equation}
where the following $i-j$ is interpreted \emph{mod} $m$.}

\emph{Proof of Claim 3:} Since $u_i\in H_{1,1+2i}$, $\D(u_i)\in H_1^l\otimes H_{1+2i}^r$. Noting that $H_1^l=\bigoplus_{j=0}^{m-1}H_0u_j$ and $H_{1+2i}^r=\bigoplus_{s=0}^{m-1}H_0g^su_{i-s}$, we can write $$\D(u_i)=\sum_{0\leqslant j,s\leqslant m-1}F^i_{js}(u_j\otimes g^su_{i-s}),$$ where $F^i_{js}\in H_0\otimes H_0$.
Then we divide the proof into two steps.

\noindent $\bullet$ \emph{Step} 1 ($\D(u_i)=\sum_{0\leqslant j\leqslant m-1}F^i_{jj}(u_j\otimes g^ju_{i-j})$).

Recall that $u_ig=\lambda_ix^{-2d}gu_i$, where $\lambda_i$ is either $\gamma^i$ for all $i$ or $-\gamma^i$ for all $i$. The equations
\begin{align*}
\D(u_ig)&=\D(u_i)\D(g)=\sum_{0\leqslant j,s\leqslant m-1}F^i_{js}(u_j\otimes g^su_{i-s})(g\otimes g)\\
           &=\sum_{0\leqslant j,s\leqslant m-1}F^i_{js}(\lambda_j x^{-2d}gu_j\otimes \lambda_{i-s} x^{-2d}g^{s+1}u_{i-s})\\
           &=\sum_{0\leqslant j,s\leqslant m-1}\lambda_j\lambda_{i-s}( x^{-2d}g\otimes x^{-2d}g)F^i_{js}(u_j\otimes g^s u_{i-s})
\end{align*}
and
\begin{align*}
\D(\lambda_i x^{-2d}gu_i)&=\lambda_i ( x^{-2d}g\otimes x^{-2d}g) \D(u_i)\D(g)\sum_{0\leqslant j,s\leqslant m-1}F^i_{js}(u_j\otimes g^su_{i-s})\\
           &=\sum_{0\leqslant j,s\leqslant m-1}\lambda_i ( x^{-2d}g\otimes x^{-2d}g)F^i_{js}(u_j\otimes g^s u_{i-s})
\end{align*}
imply that $\lambda_i=\lambda_j\lambda_{i-s}$ for all $j, s$. If $\lambda_i=-\gamma^i$ for all $i$, then we have $-\gamma^i=\lambda_i=\lambda_j\lambda_{i-s}=\gamma^{j+i-s}.$ This implies $j=s\pm m/2$. Applying $(\epsilon\otimes \id)$ to $\D(u_i)$, $$(\epsilon\otimes \id)\D(u_i)=(\epsilon\otimes \id)(F^i_{0,\; m/2})g^{m/2}u_{i-m/2}\neq u_i,$$ which is absurd. If $\lambda_i=\gamma^i$ for all $i$, then $\gamma^i=\lambda_i=\lambda_j\lambda_{i-s}=\gamma^{j+i-s}$. This implies $j=s$ (which is compatible with the equality $(\epsilon\otimes \id)\D(u_i)=u_i$). So we get $F^i_{js}\neq 0$ only if $j=s$ and $\lambda_i=\gamma^i$ for all $i$. Thus we have $\D(u_i)=\sum_{0\leqslant j\leqslant m-1}F^i_{jj}(u_j\otimes g^ju_{i-j})$ for all $i$.

\noindent $\bullet$ \emph{Step} 2 (There exist $f_{ij}, h_{ij} \in H_{0}$ with $h_{ij}$ monic such that $F^i_{jj}=f_{ij}\otimes h_{ij}$ for $0\leqslant i,j\leqslant m-1$).

We replace $F^i_{jj}$ by $F^i_{j}$ for convenience. Since
\begin{align*}
(\D\otimes \id)\D(u_i)&=(\D\otimes \id)(\sum_{0\leqslant j\leqslant m-1}F^i_{j}(u_j\otimes g^ju_{i-j}))\\
           &=\sum_{0\leqslant j\leqslant m-1}(\D\otimes \id)(F^i_{j})(\sum_{0\leqslant s\leqslant m-1}F_s^j(u_s\otimes g^su_{j-s})\otimes g^ju_{i-j})\\
           &=\sum_{0\leqslant j, s\leqslant m-1}(\D\otimes \id)(F^i_{j})(F_s^j\otimes 1)(u_s\otimes g^su_{j-s}\otimes g^ju_{i-j})
\end{align*}
and
\begin{align*}
(\id\otimes\D)\D(u_i)&=(\id\otimes\D)(\sum_{0\leqslant j\leqslant m-1}F^i_{j}(u_j\otimes g^ju_{i-j}))\\
           &=\sum_{0\leqslant j\leqslant m-1}(\id\otimes\D)(F^i_{j})(u_j\otimes(\sum_{0\leqslant s\leqslant m-1}F_s^{i-j}(g^ju_s\otimes g^{j+s}u_{i-j-s}))\\
           &=\sum_{0\leqslant j, s\leqslant m-1}(\id\otimes\D)(F^i_{s})(1\otimes F_{j-s}^{i-s})(u_s\otimes g^su_{j-s}\otimes g^ju_{i-j}),
\end{align*}
we have \begin{equation}\label{eqclaim3}(\D\otimes \id)(F^i_{j})(F_s^j\otimes 1)=(\id\otimes\D)(F^i_{s})(1\otimes F_{j-s}^{i-s})\end{equation} for all $0\leqslant i, j, s\leqslant m-1$.

Begin with the case $i=j=s=0$. Let $F_0^0=\sum_{p,q}k_{pq}x^p\otimes x^q$. Comparing equation
\begin{align*}
(\D\otimes \id)(F^0_{0})(F_0^0\otimes 1)&=(\sum_{p,q}k_{pq}x^p\otimes x^p\otimes x^q)(\sum_{p',q'}k_{p'q'}x^{p'}\otimes x^{q'}\otimes 1)\\
           &=(\sum_{p,q,p',q'}k_{pq}k_{p'q'}x^{p+p'}\otimes x^{p+q'}\otimes x^q)
\end{align*}
and equation
\begin{align*}
(\id\otimes\D)(F^0_{0})(1\otimes F_0^0)&=(\sum_{p,q}k_{pq}x^p\otimes x^q\otimes x^q)(\sum_{p',q'}k_{p'q'}1 \otimes x^{p'}\otimes x^{q'} )\\
           &=(\sum_{p,q,p',q'}k_{pq}k_{p'q'}x^{p}\otimes x^{q+p'}\otimes x^{q+q'}),
\end{align*}
one can see that $p=q=0$ by comparing the degrees of $x$ in these two expressions. Then $F_0^0= 1\otimes 1$ by applying $(\epsilon\otimes \id)\D$ to $u_0$.
Next, consider the case $j=s=0$. Write $F_0^i=\sum_{p,q}k_{pq}x^p\otimes x^q$. Similarly, we have $F_0^i= x^{a_i}\otimes 1$ for some $a_i\in \mathbb{Z}$ by the equation $$(\D\otimes \id)(F^i_{0})(F_0^0\otimes 1)=(\id\otimes\D)(F^i_{0})(1\otimes F_0^i).$$  Finally, write $F_j^i=\sum_{p,q}k_{pq}x^p\otimes x^q$ and consider the case $s=0$.
Let $F_0^i= x^{a_i}\otimes 1$ and $F_0^j=x^{a_j}\otimes 1$. The equation
\begin{align*}
(\sum_{p,q}k_{pq}x^{p+a_j}\otimes x^p\otimes x^q)=(\D\otimes \id)(F^i_{j})(F_0^j\otimes 1)=(\id\otimes\D)(F^i_{0})(1\otimes F_j^i)=(\sum_{p,q}k_{pq}x^{a_i}\otimes x^p\otimes x^q)
\end{align*}
shows that $p=a_i-a_j$, that is, $F_j^i=x^{c_{ij}}\otimes \beta_{ij}$ some $c_{ij}\in \mathbb{Z}, \beta_{ij}\in H_0$.

By steps 1 and 2, $F^i_j$ can be written as $f_{ij}\otimes h_{ij}$ with $h_{ij}$ monic after multiplying suitable scalar, where $f_{ij}, h_{ij}\in k[x^{\pm 1}]$. That is, $$\D(u_i)=\sum_{j=0}^{m-1}f_{ij}u_j\otimes
h_{ij}g^ju_{i-j},$$ where $f_{ij},h_{ij}\in k[x^{\pm 1}]$ with $h_{ij}$ monic. \qed

Since $\lambda_i=\gamma^i$ for all $i$ has been shown above, we can improve Claim 2 as

\noindent \emph{Claim 2'. We have $u_ig=\gamma^{i} x^{-2d}gu_i$ for $0\leqslant i\leqslant m-1$.}

By Claim 2', we have a unified formula in $H$: For all $s\in \mathbbm{Z}$,
\begin{equation}\label{r7} u_ig^s=\gamma^{is}x^{-2sd}g^su_i.\end{equation}

\noindent\emph{Claim 4. We have $\phi_i=1-\gamma^{-i-1}x^d$ for $0\leqslant i\leqslant m-1$.}

\emph{Proof of Claim 4:} By Claim 3, there are polynomials $f_{0j},h_{0j},f_{1j},h_{1j},$ such that
$$\D(u_0)=u_0\otimes u_0+ f_{01}u_1\otimes
h_{01}gu_{m-1}+ \cdots  + f_{0, m-1}u_{m-1}\otimes h_{0,
m-1}g^{m-1}u_1,$$ $$\D(u_1)=f_{10}u_0\otimes u_1+ u_1\otimes
h_{11}gu_{0}+ \cdots  + f_{1, m-1}u_{m-1}\otimes h_{1,
m-1}g^{m-1}u_2.$$

Firstly, we will show $\phi_0=1-\gamma^{-1}x^d$ by considering the equations
$$\D(yu_0)_{11\otimes 13}=\D(\xi x^d u_0y)_{11\otimes 13}=\D(\phi_0u_1)_{11\otimes 13}.$$
Direct computations show that
\begin{align*}
\D(yu_0)_{11\otimes 13}&=u_0\otimes yu_0+y f_{0, m-1} u_{m-1}\otimes g h_{0, m-1} g^{m-1}u_1\\
&=u_0\otimes \phi_0u_1+f_{0, m-1}\phi_{m-1}u_0\otimes x^{md} h_{0,m-1} u_1,\\
\D(\xi x^d u_0y)_{11\otimes 13}&=\xi x^du_0\otimes x^du_0y+ \xi x^d f_{0, m-1} u_{m-1}y\otimes x^d h_{0, m-1} g^{m-1}u_1 g\\
                               &=x^du_0\otimes \phi_0u_1+f_{0, m-1}\phi_{m-1}u_0\otimes \gamma x^{(m-1)d} h_{0,
                                 m-1} u_1.
\end{align*}
Owing to $\D(yu_0)_{11\otimes 13}=\D(\xi x^d u_0y)_{11\otimes 13}$,
$$(1-x^d)u_0\otimes \phi_0u_1+ f_{0, m-1}\phi_{m-1}u_0\otimes (x^d-\gamma) x^{(m-1)d} h_{0, m-1} u_1=0.$$
Thus we can assume $\phi_0=c_0 (x^d-\gamma) x^{(m-1)d} h_{0,
m-1}$ for some $0\neq c_0\in k$. Then $1-x^d=-c_0^{-1}f_{0,
m-1}\phi_{m-1}$. Therefore,
\begin{align*}
\D(yu_0)_{11\otimes 13}&=u_0\otimes \phi_0u_1-c_0(1-x^d)u_0\otimes \frac{1}{c_0}\frac{x^d}{x^d-\gamma} \phi_0 u_1\\
                       &=u_0\otimes (1-\frac{x^d}{x^d-\gamma})\phi_0u_1+x^du_0\otimes \frac{x^d}{x^d-\gamma} \phi_0u_1\\
                       &=u_0\otimes (\frac{-\gamma}{x^d-\gamma})\phi_0u_1+x^du_0\otimes \frac{x^d}{x^d-\gamma} \phi_0u_1,
\end{align*} where $\frac{1}{x^d-\gamma} \phi_0$ is understood as $c_0 x^{(m-1)d}h_{0,m-1}$.
Note that $\D(\phi_0u_1)_{11\otimes 13}=\D(\phi_0)(f_{1 0}u_0\otimes
u_1)$. Comparing the first components of $\D(yu_0)_{11\otimes 13}$
and $\D(\phi_0u_1)_{11\otimes 13}$, we get $\phi_0=1+\theta x^d$ for
some $\theta\in k$. Then it is not hard to see that $f_{1 0}=1,
f_{0, m-1}=\gamma^{-1}, h_{0, m-1}=x^{-(m-1)d}$ and
$\theta=-\gamma^{-1}$. So $\phi_0=1-\gamma^{-1}x^d$.

Secondly, we want to determine $\phi_s$ for $s\geqslant 1$. To attack this, we will prove the fact
\begin{equation}\label{r6} f_{j0}=h_{j0}=1 \end{equation}
 for all $0\leqslant j\leqslant m-1$ at the same time. We proceed by induction. We already know that $f_{00}=h_{00}=f_{1 0}=h_{1 0}=1$. Assume that $f_{i, 0}=h_{i, 0}=1$ now. Similarly, direct computations show that \begin{align*}
\D(yu_i)_{11\otimes (1, 3+2i)}&=u_0\otimes yu_i+y f_{i, m-1} u_{m-1}\otimes g h_{i, m-1} g^{m-1}u_{i+1}\\
                            &=u_0\otimes \phi_iu_{i+1}+f_{i, m-1}\phi_{m-1}u_0\otimes x^{md} h_{i, m-1} u_{i+1},\\
\D(\xi x^d u_iy)_{11\otimes (1, 3+2i)}&=\xi x^du_0\otimes x^du_iy+ \xi x^d f_{i, m-1} u_{m-1} y\otimes x^d h_{i, m-1} g^{m-1}u_{i+1} g\\
                                 &=x^du_0\otimes \phi_iu_{i+1}+f_{i, m-1}\phi_{m-1}u_0\otimes \gamma^{i+1} x^{(m-1)d} h_{i, m-1} u_{i+1}.
\end{align*}
By $\D(yu_i)_{11\otimes (1, 3+2i)}=\D(\xi x^d u_iy)_{11\otimes (1, 3+2i)}$, $$(1-x^d)u_0\otimes \phi_iu_{i+1}+ f_{i, m-1}\phi_{m-1}u_0\otimes (x^d-\gamma^{i+1}) x^{(m-1)d} h_{i, m-1} u_{i+1}=0.$$
Thus we can assume $\phi_i=c_i (x^d-\gamma^{i+1}) x^{(m-1)d} h_{i,
m-1}$ for some $0\neq c_i\in k$. Then $1-x^d=-c_i^{-1}f_{i,
m-1}\phi_{m-1}$. Therefore
\begin{align*}
\D(yu_i)_{11\otimes (1, 3+2i)}&=u_0\otimes \phi_iu_{i+1}-c_i(1-x^d)u_0\otimes \frac{1}{c_i}\frac{x^d}{x^d-\gamma^{i+1}} \phi_i u_{i+1}\\
                            &=u_0\otimes (\frac{-\gamma^{i+1}}{x^d-\gamma^{i+1}})\phi_iu_{i+1}+x^du_0\otimes \frac{x^d}{x^d-\gamma^{i+1}} \phi_iu_{i+1}.
\end{align*}
Note that $\D(\phi_iu_{i+1})_{11\otimes (1,3+2i)}=\D(\phi_i)(f_{i+1,
0}u_0\otimes h_{i+1, 0}u_{i+1})$. Comparing the first components of
$\D(yu_i)_{11\otimes (1, 3+2i)}$ and $\D(\phi_iu_{i+1})_{11\otimes
(1,3+2i)}$, we get $\phi_i=1-\gamma^{-i-1}x^d$ similarly. And it is
not hard to see that $f_{i+1, 0}=h_{i+1, 0}=1$ and $f_{i,
m-1}=\gamma^{-i-1}, h_{i, m-1}=x^{-(m-1)d}$. So we prove that
$f_{i+1, 0}=h_{i+1, 0}=1$ at the same time.\qed

\noindent \emph{Claim 5. The coproduct of $H$ is given by
$$\D(u_i)=\sum_{j=0}^{m-1}\gamma^{j(i-j)}u_j\otimes
x^{-jd}g^ju_{i-j}$$for $0\leqslant i\leqslant m-1$.}

\emph{Proof of Claim 5:} By Claim 3, $\D(u_i)=\sum_{j=0}^{m-1}f_{ij}u_j\otimes
h_{ij}g^ju_{i-j}$. So, to show this claim, it is enough to determine the explicit form of every $f_{ij}$ and $h_{ij}$. By \eqref{r6}, $f_{i,0}=h_{i,0}=1$. We will prove that $f_{ij}=\gamma^{j(i-j)}$ and $h_{ij}=x^{-jd}$ for all $0\leqslant i,
j \leqslant m-1$ by induction. So it is enough to show that
$f_{i,j+1}=\gamma^{(j+1)(i-j-1)}$ and $h_{i, j+1}=x^{-(j+1)d}$
 under the hypothesis of $f_{ij}=\gamma^{j(i-j)}$ and
 $h_{ij}=x^{-jd}$. In fact,
\begin{align*}
\D(yu_i)_{(1, 3+2j)\otimes (3+2j, 3+2i)}&=yf_{ij}u_j\otimes gh_{ij}g^ju_{i-j}+ f_{i, j+1} u_{j+1}\otimes y h_{i, j+1} g^{j+1}u_{i-j-1}\\
                                    &=f_{ij}yu_j\otimes h_{ij}g^{j+1}u_{i-j}+ f_{i, j+1} u_{j+1}\otimes \gamma^{j+1} h_{i, j+1} g^{j+1}y u_{i-j-1},\\
\D(\xi x^d u_iy)_{(1, 3+2j)\otimes (3+2j, 3+2i)}&=\xi x^df_{ij}u_jy\otimes x^dh_{ij}g^ju_{i-j}g+ \xi x^df_{i, j+1} u_{j+1}\otimes x^d h_{i, j+1} g^{j+1}u_{i-j-1} y\\
                                            &=f_{ij}yu_j\otimes \gamma^{i-j}x^{-d}h_{ij}g^{j+1}u_{i-j} + x^d f_{i, j+1} u_{j+1}\otimes  h_{i, j+1} g^{j+1}y u_{i-j-1}.
\end{align*}
By $\D(yu_i)_{(1, 3+2j)\otimes (3+2j, 3+2i)}=\D(\xi x^d u_iy)_{(1,
3+2j)\otimes (3+2j, 3+2i)}$,
$$f_{ij}yu_j\otimes (1-\gamma^{i-j}x^{-d})h_{ij}g^{j+1}u_{i-j} = (x^d-\gamma^{j+1})f_{i, j+1} u_{j+1}\otimes h_{i, j+1} g^{j+1}y u_{i-j-1}.$$
By induction, we have
\begin{align*}&\gamma^{j(i-j)}(1-\gamma^{-j-1}x^{d})u_{j+1}\otimes
(x^d-\gamma^{i-j})x^{-(j+1)d}g^{j+1}u_{i-j}\\& =
(x^d-\gamma^{j+1})f_{i, j+1} u_{j+1}\otimes
(1-\gamma^{-i+j}x^d)h_{i, j+1} g^{j+1} u_{i-j}.\end{align*}
This implies that
$h_{i, j+1}=x^{-(j+1)d}$ and
$f_{i,j+1}=\gamma^{i-j}\gamma^{-j-1}\gamma^{j(i-j)}=\gamma^{(j+1)(i-j-1)}$.\qed

\noindent \emph{Claim 6. For $0\leqslant i,j\leqslant m-1$, the multiplication between $u_i$ and $u_j$ satisfies that $$u_iu_j=(-1)^{-j}\xi^{-j}\gamma^{\frac{j(j+1)}{2}}\frac{1}{m}x^{a}\;{[i, m-2-j]}\;y^{\overline{i+j}}g$$
for some $a\in \mathbbm{Z}$ and $i+j$ is interpreted \emph{mod} $m$.}

\emph{Proof of Claim 6:} We need to consider the relation between $u_0^2$ and $u_ju_{m-j}$
for all $1\leqslant j \leqslant m-1$ at first.

By definition, $x^d\overline{\phi_s}=-\gamma^{-s-1}\phi_{m-s-2}$ for
all $s$.  Then
\begin{align*}
y^mu_0^2&=\xi^{m-j}x^{(m-j)d}y^ju_0y^{m-j}u_0=\xi^{m-j}x^{(m-j)d}\phi_0\cdots \phi_{j-1}u_j\phi_0\cdots \phi_{m-j-1}u_{m-j}\\
         &=\xi^{m-j}x^{(m-j)d}\phi_0\cdots \phi_{j-1}\cdot\overline{\phi_0}\cdots \overline{\phi_{m-j-1}}u_ju_{m-j}\\
         &=(-1)^{m-j}\xi^{m-j} \gamma^{-\frac{(m-j)(m-j+1)}{2}}\phi_0\cdots \phi_{m-2} \cdot \phi_{j-1}u_ju_{m-j}.
\end{align*}
So
\begin{equation}\label{r8}\phi_{m-1}u_0^2=(-1)^{m-j}\xi^{m-j}\gamma^{-\frac{(m-j)(m-j+1)}{2}}\phi_{j-1}u_ju_{m-j} .\end{equation}
Since $u_0^2, u_ju_{m-j}\in H_{22}=k[x^{\pm 1}]g$, we may assume
$u_0^2=\alpha_0 g, u_ju_{m-j}=\alpha_j g$ for some $\alpha_0,
\alpha_j \in k[x^{\pm 1}]$ for all $1\leqslant j \leqslant m-1$.

Then Equation \eqref{r8} implies $\alpha_0=\alpha \phi_0\cdots
\phi_{m-2}$ for some $\alpha \in k[x^{\pm 1}]$. We claim $\alpha$ is
invertible. Indeed, by
$\phi_{m-1}\alpha_0=(-1)^{m-j}\xi^{m-j}\gamma^{-\frac{(m-j)(m-j+1)}{2}}\phi_{j-1}
\alpha_j$, we have
$$\alpha_j=(-1)^{j-m}\xi^{j-m}\gamma^{\frac{(m-j)(m-j+1)}{2}}\alpha
\;{]{j-1}, {j-1}[}.$$ Then
$$H_{11}\cdot H_{11} + \sum_{j=1}^{m-1}H_{1, 1+2j}\cdot H_{1,
1+2(m-j)}\subseteq \alpha H_{22}.$$
By the strong grading of $H$,
$$H_{22}=H_{11}\cdot H_{11} + \sum_{j=1}^{m-1}H_{1, 1+2j}\cdot H_{1,
1+2(m-j)},$$  which shows that $\alpha$ must be invertible. Since
$\epsilon(\alpha_{0})=1$ and $\epsilon(\phi_0\cdots \phi_{m-2})=m$, we
may assume $\alpha_0=\frac{1}{m}x^a \phi_0\cdots \phi_{m-2}$ for
some integer $a$. Thus
\begin{align*}
u_ju_{m-j}&=(-1)^{j-m}\xi^{j-m}\gamma^{\frac{(m-j)(m-j+1)}{2}}
\frac{1}{m} x^a \;{]{j-1}, {j-1}[}\;g \\
&=(-1)^{j}\xi^{j}\gamma^{\frac{-j(-j+1)}{2}}
\frac{1}{m} x^a \;{]{j-1}, {j-1}[}\;g.
\end{align*}

\noindent\emph{Case} 1. If $0\leqslant i+j \leqslant m-2$, then
\begin{align*}
y^{i+j}u_0^2&=\xi^jx^{jd}y^iu_0y^ju_0\\
            &=\xi^jx^{jd}\phi_0\cdots \phi_{i-1}u_i\phi_0\cdots \phi_{j-1}u_j\\
            &=\xi^jx^{jd}\phi_0\cdots \phi_{i-1}\cdot\overline{\phi_0}\cdots \overline{\phi_{j-1}}u_iu_j\\
            &=(-1)^j\xi^j\gamma^{-\tfrac{j(j+1)}{2}}\phi_0\cdots \phi_{i-1}\cdot\phi_{m-1-j}\cdots \phi_{m-2}u_iu_j.
\end{align*}
So
$$u_iu_j=(-1)^{-j}\xi^{-j}\gamma^{\frac{j(j+1)}{2}}\frac{1}{m}x^a\phi_i\cdots
\phi_{m-2-j}y^{i+j}g.$$
\noindent\emph{Case} 2. If $i+j=m-1$, then
\begin{align*}
y^{i+j}u_0^2&=\xi^jx^{jd}y^iu_0y^ju_0\\
            &=\xi^jx^{jd}\phi_0\cdots \phi_{i-1}u_i\phi_0\cdots \phi_{j-1}u_j\\
            &=\xi^jx^{jd}\phi_0\cdots \phi_{i-1}\cdot\overline{\phi_0}\cdots \overline{\phi_{j-1}}u_iu_j\\
            &=(-1)^j\xi^j\gamma^{-\tfrac{j(j+1)}{2}}\phi_0\cdots \phi_{i-1}\cdot\phi_{m-1-j}\cdots \phi_{m-2}u_iu_j\\
            &=(-1)^j\xi^j\gamma^{-\tfrac{j(j+1)}{2}}\phi_0\cdots\phi_{m-2}u_iu_j.
\end{align*}
So
$$u_iu_j=(-1)^{-j}\xi^{-j}\gamma^{\frac{j(j+1)}{2}}\frac{1}{m}x^ay^{i+j}g.$$
\noindent\emph{Case} 3. If $m \leqslant i+j \leqslant 2m-2$, then
\begin{align*}
y^{i+j}u_0^2&=\xi^jx^{jd}y^iu_0y^ju_0\\
            &=\xi^jx^{jd}\phi_0\cdots \phi_{i-1}u_i\phi_0\cdots \phi_{i-1}u_j\\
            &=\xi^jx^{jd}\phi_0\cdots \phi_{i-1}\cdot\overline{\phi_0}\cdots \overline{\phi_{j-1}}u_iu_j\\
            &=(-1)^j\xi^j\gamma^{-\tfrac{j(j+1)}{2}}\phi_0\cdots \phi_{i-1}\cdot\phi_{m-1-j}\cdots \phi_{m-2}u_iu_j.
\end{align*}
So
$$u_iu_j=(-1)^{-j}\xi^{-j}\gamma^{\frac{j(j+1)}{2}}\frac{1}{m}x^a\phi_i \cdots \phi_{m-1}\phi_0\cdots
\phi_{m-2-j}y^{i+j-m}g.$$

Using the notations introduced in Subsection 4.1, we have a unified
expression:
\begin{align*}
u_iu_j&=(-1)^{-j}\xi^{-j}\gamma^{\frac{j(j+1)}{2}}\frac{1}{m}x^{a}\;{[i, m-2-j]}\;y^{\overline{i+j}}g\\
      &=(-1)^{-j}\xi^{-j}\gamma^{\frac{j(j+1)}{2}}\frac{1}{m}x^{a}\;{]{-1-j},{i-1}[}\;y^{\overline{i+j}}g
\end{align*}
for all $i, j$. \qed

\noindent \emph{Claim 7. We have $\xi^2=\gamma, \ a=-\frac{1+m}{2}d$ and
$$S(u_i)=(-1)^i \xi^{-i}\gamma^{-\frac{i(i+1)}{2}} x^{id+\frac{3}{2}(1-m)d}
g^{m-1-i}u_i$$ for $0\leqslant i\leqslant m-1.$}

\emph{Proof of Claim 7:}
Since $S(H_{ij})=H_{-j, -i}$,  $S(u_0)=hg^{m-1}u_0$ for some
$h\in k[x^{\pm 1}]$. Combining
\begin{align*}
S(yu_0)&=S(u_0)S(y)=hg^{m-1}u_0 \cdot (-yg^{-1})=-\xi x^{-d}hg^{m-1}yu_0g^{-1}\\
       &=-\xi x^{-d}\phi_0 hg^{m-1}u_1g^{-1}=-\xi \gamma^{-1} x^{d}\phi_0 h g^{m-2} u_1
\end{align*} with $$S(yu_0)=S(\phi_0 u_1)=S(u_1)S(\phi_0)=\phi_0 S(u_1),$$
we get $S(u_1)=-\xi \gamma^{-1} x^{d} h g^{m-2} u_1$. The computation above tells us that we can prove that
$$S(u_i)=(-1)^i \xi^{-i}\gamma^{-\frac{i(i+1)}{2}} x^{id} h
g^{m-1-i}u_i$$ by induction. In fact, assume that $S(u_i)=(-1)^i \xi^{-i}\gamma^{-\frac{i(i+1)}{2}} x^{id} h
g^{m-1-i}u_i$, by combining
\begin{align*}
S(yu_i)&=S(u_i)S(y)=(-1)^i \xi^{-i}\gamma^{-\frac{i(i+1)}{2}}x^{id}hg^{m-1-i}u_i (-yg^{-1})\\
       &=\phi_i\cdot(-1)^{i+1} \xi^{-(i+1)}\gamma^{-\frac{(i+1)(i+2)}{2}}x^{(i+1)d} h g^{m-2-i}u_{i+1}
\end{align*} with $$S(yu_i)=S(\phi_i u_{i+1})=S(u_{i+1})S(\phi_i)=\phi_i S(u_{i+1}),$$
we find that $S(u_{i+1})=(-1)^{i+1} \xi^{-(i+1)}\gamma^{-\frac{(i+1)(i+2)}{2}}
x^{(i+1)d} h g^{m-2-i}u_{i+1}$.

In order to determine the relationship between $\xi$ and $\gamma$,
we consider the equality $(\Id*S)(u_1)=0$. By computation,
\begin{align*}
(\Id*S)(u_1)&=\sum_{j=0}^{m-1}\gamma^{j(1-j)}u_jS(x^{-jd}g^ju_{1-j}) \\
           &=\sum_{j=0}^{m-1}\gamma^{j(1-j)}u_j\cdot (-1)^{1-j}\xi^{j-1}\gamma^{-\frac{(1-j)(2-j)}{2}}x^{(1-j)d}hg^{m-2+j}u_{1-j}g^{-j}x^{jd}\\
           &=\sum_{j=0}^{m-1}(-1)^{1-j}\xi^{j-1}\gamma^{j(1-j)-\frac{(1-j)(2-j)}{2}}x^{(2j-1)d}\overline{h}u_jg^{m-2+j}u_{1-j}g^{-j}\\
           &=\sum_{j=0}^{m-1}(-1)^{1-j}\xi^{j-1}\gamma^{-\frac{(1-j)(2-j)}{2}-2j}x^{(3-2m)d}\overline{h}g^{m-2}u_ju_{1-j}\\
           &=\frac{1}{m}\xi^{-2}x^{(3-2m)d+a}\overline{h}g^{m-1}(
           \sum_{j=0}^{m-1}\xi^{2j}\gamma^{-2j}\;{]{j-2}, {j-1}[}),
\end{align*} where Equation \eqref{r7} is used.
Thus $$(\Id*S)(u_1)=0 \Leftrightarrow
\sum_{j=0}^{m-1}\xi^{2j}\gamma^{-2j}\;{]{j-2}, {j-1}[}=0.$$
This forces $\xi^2=\gamma$ by Corollary \ref{ce3}.

Next, we will show $h=x^{\frac{3}{2}(1-m)d}$ by the equations
$$(S*\Id)(u_0)=(\Id*S)(u_0)=1.$$ Indeed,
\begin{align*}
(S*\Id)(u_0)&=\sum_{j=0}^{m-1}S(\gamma^{-j^2}u_j)x^{-jd}g^ju_{-j} \\
           &=\sum_{j=0}^{m-1}\gamma^{-j^2} (-1)^j \xi^{-j} \gamma^{-\frac{j(j+1)}{2}}x^{jd} h g^{m-j-1} u_jx^{-jd}g^ju_{-j} \\
           &=hg^{m-1}(\sum_{j=0}^{m-1}(-1)^j \xi^{-j}\gamma^{-\frac{j(j+1)}{2}}u_ju_{-j})\\
           &=hg^{m-1}(\sum_{j=0}^{m-1}(-1)^j \xi^{-j}\gamma^{-\frac{j(j+1)}{2}} (-1)^{j-m}\xi^{j-m}\gamma^{\frac{1}{2}(m-j)(m-j+1)}
             \frac{1}{m}x^a \;{]{j-1}, {j-1}[}\; g  )\\
           &=\frac{1}{m}x^a h g^m(\sum_{j=0}^{m-1}(-1)^{-m} \xi^{-m}\gamma^{\frac{m(m+1)}{2}-j}\;{]{j-1}, {j-1}[})\\
           &=\frac{1}{m}x^a h g^m(\sum_{j=0}^{m-1}\gamma^{-j}\;{]{j-1}, {j-1}[})\\
           &=x^{(2m-1)d+a}h \quad (\,\textrm{by Lemma}\ \ref{ce1}),\\
(\Id*S)(u_0)&=\sum_{j=0}^{m-1}\gamma^{-j^2}u_j\cdot S(x^{-jd}g^ju_{-j}) \\
           &=\sum_{j=0}^{m-1}\gamma^{-j^2}u_j\cdot S(u_{-j})S(g^j)x^{jd} \\
           &=\sum_{j=0}^{m-1}\gamma^{-j^2}u_j \cdot (-1)^{-j}\xi^{j}\gamma^{\frac{j(1-j)}{2}}x^{-jd}hg^{m+j-1}u_{-j}g^{-j}x^{jd}\\
           &=\sum_{j=0}^{m-1}\gamma^{-j^2}u_j \cdot (-1)^{-j}\xi^{j}\gamma^{\frac{j(1-j)}{2}+j^2}hg^{m-1}u_{-j}\\
           &=x^{(2-2m)d}\overline{h}g^{m-1}(\sum_{j=0}^{m-1}(-1)^{-j}\xi^{j}\gamma^{\frac{j(1-j)}{2}-j}u_ju_{-j})\\
           &=\frac{1}{m}x^{(2-m)d+a}\overline{h}(\sum_{j=0}^{m-1}\xi^{2j}\gamma^{-j}\;{]{j-1}, {j-1}[})\\
           &=x^{(2-m)d+a}\overline{h} \quad (\,\textrm{by the proof of Lemma}\ \ref{ce1}).
\end{align*}
So, $(S*\Id)(u_0)=(\Id*S)(u_0)=1$ implies
$h=x^{(1-2m)d-a}=x^{(2-m)d+a}$. Thus $h=x^{\frac{3}{2}(1-m)d}$ and
$a=-\frac{1+m}{2}d$.  Therefore, for $0\leqslant i\leqslant m-1$,
$$S(u_i)=(-1)^i \xi^{-i}\gamma^{-\frac{i(i+1)}{2}} x^{id+\frac{3}{2}(1-m)d}
g^{m-1-i}u_i.$$  \qed

From Claim 7, we find that $2|(1+m)d$ and we can improve Claim 6 as the following form:

\noindent \emph{Claim 6'.  For $0\leqslant i,j\leqslant m-1$, the multiplication between $u_i$ and $u_j$ satisfies that $$u_iu_j=(-1)^{-j}\xi^{-j}\gamma^{\frac{j(j+1)}{2}}\frac{1}{m}x^{-\frac{1+m}{2}d}\;{[i, m-2-j]}\;y^{\overline{i+j}}g$$
where $i+j$ is interpreted \emph{mod} $m$.}

We can prove Proposition \ref{p8.2} now. By Claims 1,2',3,4,5,6' and 7, we have a natural surjective Hopf homomorphism
$$f:\; D(m,d,\xi)\to H,\;x\mapsto x,\ y\mapsto y,\ g\mapsto g,\ u_{i}\mapsto u_{i}$$
for $0\leqslant i\leqslant m-1$. It is not hard to see that $f|_{D_{ij}}:\; D_{ij}\to H_{ij}$ is an isomorphism of
$k[x^{\pm 1}]$-modules for $0\leqslant i,j\leqslant 2m-1$. So $f$ is an isomorphism. \qed

We conclude this paper by giving the classification of  prime regular Hopf algebras of GK-dimension one.

\begin{theorem}\label{t8.3} Let $H$ be a prime regular Hopf algebra of GK-dimension one. Then it is isomorphic to one of the following:

\emph{(1)} the Hopf algebras listed in Subsection 2.3;

\emph{(2)} the Hopf algebras constructed in Subsection 4.1.
\end{theorem}
\begin{proof} By Theorem \ref{t2.7}, we only need to consider the case $\io(H)>\im(H)>1$. In this case, $\widetilde{H}$
can be constructed. By Proposition \ref{t3.9}, $\widetilde{H}$ is either primitive or group-like. If $\widetilde{H}$ is primitive, then
$H$ is isomorphic to an infinite dimensional Taft algebra by Theorem \ref{t4.10}. If $\widetilde{H}$ is group-like,  owing to Proposition \ref{p4.11} there is no such $H$ satisfying $\frac{\io(H)}{\im(H)}>2$. Moreover, if $\frac{\io(H)}{\im(H)}=2$ then $H$ is one of the Hopf algebras constructed in Subsection 4.1 by Proposition \ref{p8.2}.
\end{proof}

\section*{Acknowledgment}
We would like thank the referee for his/her very valuable comments which improve the paper greatly.

\end{document}